\documentclass[11pt,a4paper,reqno]{amsart}

\usepackage[all]{xy}
\usepackage{amsthm}
\usepackage{amssymb} 
\usepackage{amsmath,amscd} 
\usepackage{mathrsfs}
\usepackage{color}
\usepackage{enumerate}
\usepackage{longtable}

\def\supp{\mathop{\mathrm{supp}}\limits}

\def\GrMod{\operatorname{\mathsf{GrMod}}}

\def\turn!{\textup{!`}}

\def\op{\textup{op}}

\def\soc{\operatorname{soc}}

\def\tridim{\operatorname{\mathsf{tridim }}}

\def\grCM{\operatorname{\mathsf{CM}}^{\Bbb{Z}}}

\def\grmod{\operatorname{mod}^{\Bbb{Z}}}

\def\mrb{\mathrm{b}}

\def\Sing{\operatorname{Sing}}

\def\stabgrCM{\operatorname{Sing}^{\ZZ}}

\def\stabCM{\operatorname{\underline{\mathsf{CM}}}}
\def\stabgrCM{\operatorname{\underline{\mathsf{CM}}}^{\Bbb{Z}}}

\def\stabgrmod{\operatorname{\underline{mod}}^{\Bbb{Z}}}

\def\thick{\mathop{\mathsf{thick}}\nolimits}

\def\ZZ{{\Bbb Z}}

\def\sfD{{\mathsf{D}}}

\def\sfK{{\mathsf{K}}}

\def\sfT{{\mathsf{T}}}

\def\tuD{{\textup{D}}}

\def\op{\operatorname{op}}

\def\mod{\operatorname{mod}}

\def\GrMod{\operatorname{Mod}^{\mathbb{Z}}}

\def\proj{\operatorname{proj}} 
\def\inj{\operatorname{inj}}

\def\grproj{\operatorname{proj}^{\mathbb{Z}}} 
\def\grinj{\operatorname{inj}^{\mathbb{Z}}}

\def\add{\operatorname{add}}

\def\image{\operatorname{Im}}

\def\Hom{\operatorname{Hom}}

\def\stabHom{\operatorname{\underline{Hom}}}

\def\End{\operatorname{End}}

\def\Ext{\operatorname{Ext}}

\def\gldim{\operatorname{gldim}}

\newcommand{\lotimes}{\otimes^{\Bbb{L}}}

\newtheorem{lemma}{Lemma}[section]
\newtheorem{proposition}[lemma]{Proposition}
\newtheorem{theorem}[lemma]{Theorem}
\newtheorem{corollary}[lemma]{Corollary}

\newtheorem{problem}[lemma]{Problem}

\theoremstyle{definition}
\newtheorem{remark}[lemma]{Remark}
\newtheorem{example}[lemma]{Example}

\newtheorem{definition}[lemma]{Definition}

\newtheorem{question}[lemma]{Question}

\theoremstyle{remark}

\newcommand{\qv}{\mathsf{qv}}

\title[Tilting theory for finite dimensional $1$-Iwanaga-Gorenstein algebras]
{ 
Tilting theory for finite dimensional $1$-Iwanaga-Gorenstein algebras \\
}

\author[Kimura]{Yuta Kimura}
\address{Y. Kimura: Osaka Central Advanced Mathematical Institute, Osaka Metropolitan University, 3-3-138 Sugimoto, Sumiyoshi-ku Osaka 558-8585, Japan
}
\email{yutakimura@omu.ac.jp}

\author[Minamoto]{Hiroyuki Minamoto}
\address{H. Minamoto: Department of Mathematics, Graduate School of Science/ Faculty of Science, Osaka Metropolitan University
1-1 Gakuen-cho, Nakaku, Sakai, Osaka 599-8531, Japan}
\email{minamoto@omu.ac.jp}

\author[Yamaura]{Kota Yamaura}
\address{K. Yamaura: Graduate Faculty of Interdisciplinary Research, Faculty of Engineering, University of Yamanashi, 4-4-37, Takeda, Kofu, Yamanashi 400-8510, Japan}
\email{kyamaura@yamanashi.ac.jp}


\begin{document}

\maketitle

\begin{abstract}

In representation theory of graded Iwanaga-Gorenstein algebras, tilting theory of the stable category $\stabgrCM A$ of graded Cohen-Macaulay modules plays a prominent role. 
In this paper we study the following two central problems of tilting theory of $\stabgrCM A$ in the case where $A$ is finite dimensional:
(1) Does $\stabgrCM A$ have a tilting object?
(2) Does the endomorphism algebras of tilting objects in $\stabgrCM A$ have finite global dimension?

To the problem (2) we give the complete answer. 
We show that the endomorphism algebra of any tilting object in $\stabCM^{\ZZ}A$  has finite global dimension. 
To the problem (1) we give a partial answer. 
For this purpose, first we introduce an invariant $g(A)$ for a finite dimensional graded algebra $A$. Then, we prove that in the case where $A$ is 1-Iwanaga-Gorenstein, an inequality for $g(A)$ gives a sufficient condition that a specific Cohen-Macaulay module $V$ becomes a tilting object in the stable category. 

As an application, we study the existence of tilting objects in $\stabgrCM \Pi(Q)_w$
where $\Pi(Q)_w$ is the truncated preprojective algebra of a quiver $Q$ associated to $w\in W_Q$.
We prove that if the underling graph of $Q$ is tree, then $\stabgrCM \Pi(Q)_w$ has a tilting object. 
\end{abstract}

\bigskip

\section{Introduction}

In representation theory of graded Iwanaga-Gorenstein algebras, 
tilting theory of the stable category $\stabgrCM A$ of graded Cohen-Macaulay modules have been playing a prominent role.  
Recently it became to get much wider attention, as it found connections to theory of cluster categories and mirror symmetry, see e.g., 
\cite{AIR, Futaki-Ueda, Hanihara-Iyama, HIMO, Iyama-Takahashi, KLM, LP, Mori-Ueyama, Nakajima} and a survey article \cite{Iyama-ICM}.

In this paper we deal with  a \emph{finite dimensional} 
(non-negatively) graded Iwanaga-Gorenstein algebra $A=\bigoplus_{i=0}^{\ell} A_{i}$ 
and discuss the fundamental problems of tilting theory of  $\stabgrCM A$ posed in Problem \ref{problem-1} below. 
So from now on we denote by $A=\bigoplus_{i=0}^{\ell} A_{i}$  a finite dimensional graded algebra with $A_{\ell} \neq 0$ 
and recall the basic notions. 

A (non-negatively) graded algebra $A=\bigoplus_{i=0}^{\ell} A_i$ is called a \emph{graded $d$-Iwanaga-Gorenstein algebra} (graded $d$-IG-algebra, for short) if  the graded injective dimensions of $A_A$ and ${}_AA$ are at most $d$. 
A finitely generated graded $A$-module $M$ is called \emph{Cohen-Macaulay} (CM, for short) if it satisfies 
$\Ext^{i}_A(M,A)=0$ for all $i>0$. 
Let $\grCM A$ be the category of graded Cohen-Macaulay $A$-modules, namely
\[
\grCM A:=\{ M \in \grmod A \ | \ \Ext^{i}_A(M,A)=0  \mbox{ for all } i> 0 \},
\]
where $\grmod A$ is the category of finitely generated graded $A$-modules. 
A fundamental fact  is that it is a Frobenius category, and so its  stable category 
\[
\stabgrCM A:=\grCM A/[\proj^{\ZZ}A]
\]
has a natural structure of a triangulated category \cite{Happel book}. 

Therefore, as is about any triangulated categories,  
one of the main concern is  whether  $\stabgrCM A$ has a tilting object $T$ or not. 
If the answer is affirmative, 
then there exists an equivalence $\stabgrCM A \cong\mathsf{K}^{\mrb}(\proj\Gamma)$
where we set $\Gamma := \End_{\stabgrCM A}(T)$. 
Moreover, if $\Gamma$ is of finite global dimension, 
then the homotopy category $\mathsf{K}^{\mrb}(\proj \Gamma)$ coincides with the derived category $\sfD^{\mrb}(\mod\Gamma)$, 
and consequently   
the above equivalence becomes   $\stabgrCM A \cong  \sfD^{\mrb}(\mod\Gamma)$.
Thus, the following problems are fundamental for representation theory of IG-algebras.

\begin{problem}\label{problem-1}
Let $A$ be a finite dimensional graded $d$-IG-algebra. 
\begin{enumerate}[\rm (1)]
\item 
Does $\stabgrCM A$ have a tilting object?
\item
Does the endomorphism algebras of tilting objects in $\stabgrCM A$ have finite global dimensions?
\end{enumerate}
\end{problem}

\subsection{Main results}

\subsubsection{Fineness of global dimension of the endomorphism algebra of a tilting object}

One of our main result gives the complete answer of Problem \ref{problem-1}(2).

\begin{theorem}\label{Q(c)<->(d)}
Let $A$ be a finite dimensional non-negatively graded $d$-IG-algebra.
If there exists a tilting object $T$ in $\stabgrCM A$,
then its endomorphism algebra $\Gamma:=\End_{\stabgrCM A}(T)$ has finite global dimension. 
Consequently, we obtain an equivalence $\stabgrCM A \cong \sfD^{\mrb}(\mod \Gamma)$ of triangulated categories.  
\end{theorem}

%

\subsubsection{Existence of a silting object $V$}

Next, we investigate Problem \ref{problem-1}(1) in the case $d\leq1$. 
First we recall that the third author solved Problem 1.1(1) in the case $d=0$ in \cite{Yamaura}.
He proved that if $A=\bigoplus_{i=0}^{\ell}A_{i}$ is a $0$-IG algebra ($=$a graded self-injective algebra),
then the stable category $\stabgrCM A=\stabgrmod A$ has a tilting object
if and only if $A_0$ has finite global dimension. 
In the proof of ``if part'', he considered the graded $A$-module 
\begin{eqnarray}
U:=\bigoplus_{i=1}^{\ell}A(i)_{< 0}     \label{def-U}
\end{eqnarray}
and proved that $U$ becomes a tilting object in $\stabgrmod A$.

Now we  deal with the case $d=1$. 
So let $A=\bigoplus_{i=0}^{\ell}A_i$ be a graded $1$-IG algebra. 
First note the grade $A$-module $U$ is not CM in general. 
However, since $A$ is $1$-IG,  
taking the first graded syzygy  
\begin{eqnarray}
V:=\Omega_A(U)=\bigoplus_{i\geq1}A(i)_{\geq0}  \label{def-T}
\end{eqnarray}
of $U$, we obtain a  graded CM $A$-module $V$.
It was shown that $V$ is a tilting object for certain classes of graded $1$-IG-algebras whose $0$-th subrings have finite global dimensions \cite{Kimura 1,Lu,Lu-Zhu}.
So, the following question naturally arises.

\begin{question}\label{question-syzygy-U-tilting}
Let $A$ be a finite dimensional graded $1$-IG-algebra.
Assume that $A_{0}$ has finite global dimension. 
Does $V$ become a tilting object in $\stabgrCM A$?
\end{question}

Our second result asserts that $V$ becomes a silting object which is a more general notion than tilting objects.

\begin{theorem}[{\cite[Theorem 3.0.3]{Lu-Zhu}}]\label{silting theorem}
Let $A$ be a finite dimensional non-negatively graded $1$-IG-algebra.
The following assertions hold.
\begin{enumerate}[\em (1)]
\item
$V$ is a presilting object in $\stabgrCM A$.
\item
If $A_{0}$ has finite global dimension, then $V$ is a silting object in $\stabgrCM A$.
\end{enumerate}
\end{theorem}

\begin{remark}\label{remark-Lu-Zhu}
Here we mention that M. Lu and B. Zhu studied tilting theory for graded $1$-IG-algebras in \cite{Lu-Zhu}.
In particular, they had three results related to Question \ref{question-syzygy-U-tilting} and Theorem \ref{silting theorem} as below.
\begin{enumerate}[\rm(1)]
\item 
For a graded $d$-IG algebra $B=\bigoplus_{i=0}^{\ell} B_{i}$ where $d$ is an arbitrary natural number, they proved that $\gldim B_{0} < \infty$  if $\stabgrCM B$ has a tilting object \cite[Lemma 3.0.1]{Lu-Zhu}. 
Thus, it is necessary to assume $\gldim A_{0} < 0$ in Question \ref{question-syzygy-U-tilting}. 
\item
They obtained Theorem \ref{silting theorem} independently. 
\item
They gave an example of graded $1$-IG algebra $A$ such that $A_0$ has finite global dimension and the graded $A$-module $V$ is not a tilting object.
We will also give such an example which is different from theirs. 
\end{enumerate}
\end{remark}

\subsubsection{A class of graded $1$-IG-algebras such that $V$ becomes a tilting object.}

By Remark \ref{remark-Lu-Zhu}(3), the answer of Question \ref{question-syzygy-U-tilting} is negative.
In spite of this conclusion, $V$ often becomes a tilting object.
In fact, there are two large classes of graded $1$-IG-algebras such that $V$ gives a tilting object.
One class is explained in \ref{subsubsection-ASG}, which are infinite dimensional IG-algebras, called Artin-Schelter-Gorenstein algebras (AS-Gorenstein, for short).
Note that, in the study of tilting theory of AS-Gorenstein algebras, the $a$-invariants of the algebras play an important role, see \ref{subsubsection-ASG}.
The other class is explained as below, which are finite dimensional IG-algebras called homologically well-graded IG-algebras.

In  \cite{cotilting type}, the second and the third authors studied homologically well-graded IG-algebras. 
Let $A=\bigoplus_{i=0}^{\ell}A_i$ be a finite dimensional graded $1$-IG-algebra.
We take the minimal graded injective resolution
\[
0\to A_A \to I^0 \to I^1 \to 0
\]
of $A$, and put $I:=I^0\oplus I^1$. 
If the socle $\soc(I)$ of $I$ is concentrated in degree $\ell$, the highest possible degree, 
$A$ is called \emph{homologically well-graded}. 
This condition can be written in terms of $\Ext_{A}^{n}(A/J_A,A)$, namely, 
$A$ is  homologically well-graded if and only if $\Ext_{A}^{n}(A/J_A,A)$ is concentrated in degree $\ell$ for $n=0,1$.
It follows from one of main result \cite[Theorem 1.3]{cotilting type} that 
if $A$ is a homologically well-graded $1$-IG algebra with $\gldim A_{0} < \infty$, then 
$V$ is a tilting object in $\stabgrCM A$.  


%

The result as above and the result in \ref{subsubsection-ASG} suggest that for a graded $1$-IG-algebra $A$, 
the degrees in which the extension groups $\Ext_{A}^{n}(S,A)$ from a graded simple module $S$ to $A$ lives are important for $V$ to be a tilting object.
From this viewpoint, we introduce a new invariant $g(A)$ of a finite dimensional graded algebra $A$, which can be looked as a variant of the $a$-invariant.
Then, we show that it provides a sufficient condition that $V$ becomes a tilting object.  

Although our invariant $g(M)$ is defined for a graded module $M$ over a finite dimensional graded algebra which is not necessarily IG, 
we give the definition for a graded $1$-IG algebra $A$ in this introduction.
Again, let $A=\bigoplus_{i=0}^{\ell}A_i$ be a finite dimensional graded $1$-IG-algebra.
Now, we define an integer $g(A)$ as follows: 
\[
g(A) = \sup\{ i \in \ZZ \ | \ \Ext_{A}^{n}(A/J_A, A)_{-i} \neq 0 \textup{ for some } n=0,1 \}.
\]
It follows from this definition that 
\[
\{ i\in\ZZ \ | \ \Ext_{A}^{n}(A/J_A, A)_{i} \neq 0 \textup{ for some } n=0,1 \} \subset \left[\, -g(A), \ell\, \right].
\]
Note that $A$ is homologically well-graded if and only if $g(A)=-\ell$. 
Our third result asserts that the answer of Question \ref{question-syzygy-U-tilting} is affirmative, if $g(A)$ are non-positive.

\begin{theorem}\label{tilting theorem 1}
Let $A$ be a finite dimensional non-negatively graded $1$-IG algebra such that $A_0$ has finite global dimension.
If $g(A)\leq 0$, then $V$ is a tilting object in $\stabgrCM A$. 
\end{theorem}

We remark that under the setup of Theorem \ref{tilting theorem 1}, 
the complement condition  $g(A) >0$ dose not implies $V$ is not a tilting object. 
Both cases occur that $V$ is  a tilting object or not (see Example \ref{example-g(A)=1}).

\subsubsection{A class of infinite dimensional graded $1$-IG-algebras with tilting objects.} \label{subsubsection-ASG}

Here we explain known results of tilting theory for Artin-Schelter-Gorenstein algebras.
It has been proved that the $a$-invariant gives a characterization of existence of a tilting object.

In \cite{BIY}, Buchweitz, Iyama and the third author studied tilting theory for commutative graded $1$-IG-algebras. 
Let $R=\bigoplus_{i=0}^{\infty}R_{i}$ be a  commutative graded $1$-IG-algebra such that 
the degree $0$-part $R_{0}$ coincides with the base field $K$ (i.e., $R$ is connected). 
The crucial point here is that in the commutative case $R$ satisfies the \emph{Artin-Schelter-Gorenstein} condition, 
that is that $\Ext_{R}^{n}(K,R) = 0$ for $n=0,1, \dots, d-1$ and $\Ext_{R}^{d}(K,R)=K$ as ungraded module 
(with $d=1$ in the situation of \cite{cotilting type}). 
The negative of degree in which the simple graded $R$-module $\Ext_{R}^{d}(K,R)$ concentrated is called the \emph{$a$-invariant} \cite{Goto-Watanabe}. 
The authors of \cite{BIY} dealt with the subcategory $\grCM_{0}R$ of $\grCM R$ consisting of CM-modules 
which are projective outside of the closed point. 
They proved that an analogue of $V$ gives a silting object in the stable category $\stabgrCM_{0}R$ and moreover that it becomes a tilting object precisely when the $a$-invariant is non-negative or $R$ is regular. 

Recently Iyama, Ueyama and the first author \cite{IKU} generalize these results into the case where $R$ is not necessarily commutative and connected. 
However the still AS-Gorenstein condition plays a crucial role. 
We mention that in the case where $R$ is not connected, the AS-Gorenstein condition is formulated in terms of $\Ext_{A}^{n}(S,R)$ 
where $S$ is a  graded simple $R$-modules concentrated in $0$-th degree.

\subsection{Application of Theorem \ref{tilting theorem 1} to truncated preprojective algebras}

Finally, we investigate an existence of tilting objects for a certain class of $1$-IG-algebras.
For a finite acyclic quiver $Q$ and an element $w$ of the Coxeter group $W_Q$ of $Q$, Buan-Iyama-Reiten-Scott \cite{BIRSc} constructed a graded $1$-IG-algebra $\Pi(Q)_w$.
%
The first author has studied  an existence of tilting objects in $\stabgrCM \Pi(Q)_w$.
He gave constructions of tilting objects when $w$ is $c$-sortable \cite{Kimura 1}, and when $w$ is $c$-starting or $c$-ending \cite{Kimura 2}.
Moreover, he calculated the endomorphism algebras of his tilting objects, and showed that those global dimensions are at most two.
Therefore, $\stabgrCM \Pi(Q)_w$ are equivalent to the derived categories of modules over finite dimensional algebras.

In this paper, we study existence of tilting objects in $\stabgrCM \Pi(Q)_w$ by applying Theorem \ref{tilting theorem 1}.
Our final result asserts that a tilting object in $\stabgrCM \Pi(Q)_w$ exists if the underlying graph of $Q$ is tree.

\begin{theorem}[Theorem \ref{tilting theorem}]\label{tree-tilting}
Let $Q$ be a finite acyclic quiver whose underlying graph is tree.
Then, for any $w\in W_Q$, $\stabgrCM\Pi(Q)_w$ has a tilting object. 
\end{theorem}

We note that $\Pi(Q)_w$ does not satisfy $g(\Pi(Q)_w)\leq 0$ in general, and so we could not apply Theorem \ref{tilting theorem 1} directly. 
To circumvent this difficulty, we show that the category $\stabgrCM\Pi(Q)_w$ is invariant under sink reflections of $Q$.
We also show that  $Q$ can be changed  to a quiver $Q'$ via an iterated sink reflections, such that  $g(\Pi(Q')_w)\leq 0$.
Now we can apply  Theorem \ref{tilting theorem 1} to $\stabgrCM\Pi(Q)_w\cong \stabgrCM \Pi(Q')_w$ and  obtain  the desired conclusion. 

\smallskip
\smallskip

This paper is organized as follows.
In Section 2, we recall notions and facts which will be needed in the proofs of main results.
In Section 3, we give proofs of Theorem \ref{Q(c)<->(d)}, Theorem \ref{silting theorem} and Theorem \ref{tilting theorem 1}.
In Section 4, we recall the definition of $\Pi(Q)_w$ and give a proof of Theorem \ref{tree-tilting}.

\subsection{Conventions}
We fix a base field $K$.
An algebra means a finite dimensional $K$-algebra. 
For an algebra $\Lambda$, we deal with finitely generated right $\Lambda$-modules.
We denote by $\mod\Lambda$ the category of finitely generated right $\Lambda$-modules, 
$\proj\Lambda$ the full subcategory  of $\mod\Lambda$ whose objects are projective $\Lambda$-modules. 
Also when we consider $\Lambda$-$\Lambda$-bimodules, we assume that they are finite dimensional.

A $\ZZ$-graded algebra $A=\bigoplus_{i\in\ZZ}A_i$ is called non-negatively graded if $A_i=0$ for all $i<0$.
In this paper, graded algebras means finite dimensional non-negatively graded algebras.
For a graded algebra $A$, we deal with finitely generated right graded $A$-modules $X=\bigoplus_{i\in\ZZ}X_i$.
For a graded $A$-module $X$ and $n\in\ZZ$, we denote by $X(n)$ the $n$-th degree shift of $X$, i.e. $X(n):=X$ as an $A$-module and the grading is defined by $X(n)_i:=X_{i+n}$.
For a graded $A$-module $X$ and $n\in\ZZ$, we define a graded $A$-submodule $X_{\geq n}$ of $X$ by $X_{\geq n}:=\bigoplus_{i\geq n}X_i$.
Also, we define a graded factor $A$-module $X_{<n}$ of $X$ by $X_{<n}:=X/X_{\geq n}$.
We will consider the socle $\soc (X)$ of a graded $A$-module $X$.
Note that $\soc (X)$ is a graded submodule of $X$. %

We denote by $\GrMod A$ the category of all right graded $A$-modules, and by $\grmod A$ the full subcategory of $\GrMod A$ consisting of finitely generated right graded $A$-modules.
For $X,Y\in \GrMod A$, the morphism space from $X$ to $Y$ is 
\[
\Hom_A(X,Y)_0:=\{ f\in\Hom_{A}(X,Y) \ | \  f(X_i) \subset Y_i \mbox{ for all $i\in\ZZ$ } \}.
\]
It is known that  the following equality holds if $X$ is finitely generated.
\[
\Hom_A(X,Y)=\bigoplus_{i\in \ZZ}\Hom_A(X,Y(i))_0.
\]
In particular, $\End_A(X)$ has a structure of $\ZZ$-graded algebras with $\End_A(X)_i=\Hom_A(X,X(i))_0$.
Also $\Hom_A(X,Y)$ can be regarded as a $\ZZ$-graded  $\End_A(X)$-module with $\Hom_A(X,Y)_i=\Hom_A(X,Y(i))_0$.

We denote by $\proj^{\ZZ}A$ (resp. $\inj^{\ZZ}A$) the full subcategory of $\grmod A$ whose objects are graded projective (resp. injective) $A$-modules. 
For a subset $L$ of $\ZZ$, we define a full subcategory $\mod^{L} A$ of $\grmod A$ by 
\[
\mod^{L}A=\{M \in\grmod A \ | \ M_i=0 \mbox{ for all } i\not\in L \}.
\]
We refer to \cite{Green-Gordon1} for basic facts about graded modules over graded artin algebras.

Let $A$ be a graded IG-algebra.
For $X,Y\in\stabgrCM A$, we denote by $\underline{\Hom}_A(X,Y)_0$ the morphism space from $X$ to $Y$ in $\stabgrCM A$.

For an additive category $\mathcal{A}$, we denote by $\sfK^{\mrb}(\mathcal{A})$ the homotopy category of bounded complexes over $\mathcal{A}$.
For an abelian category $\mathcal{A}$, we denote by $\sfD^{\mrb}(\mathcal{A})$ the bounded derived category of  $\mathcal{A}$.

\bigskip

\noindent
\textbf{Acknowledgement.}
The authors thank Osamu Iyama for his valuable comments and for informing us that Theorem \ref{silting theorem} was proved in the paper \cite{Lu-Zhu}.
They also grateful to Ming Lu and Bin Zhu for exchanging information about this subject with us. 

The first author is supported by Grant-in-Aid for JSPS Fellows JP22J01155,
and was partially supported by the Alexander von Humboldt Foundation in
the framework of an Alexander von Humboldt Professorship endowed by the
German Federal Ministry of Education and Research.
The second author was partially supported by JSPS KAKENHI Grant Number JP21K03210. 
The third author was supported by JSPS KAKENHI Grant Numbers JP26800007, JP18K13387, JP21K03160.

\section{Preliminaries}

In this section, we collect basic notions and facts which will be used in later sections.

\subsection{Tilting theory for algebraic triangulated categories}\label{subsec-general-tilting-theory}
We recall tilting theory for algebraic triangulated categories.

\begin{definition}
Let $\mathscr{T}$ be a triangulated category.
For an object $T\in\mathscr{T}$, we denote by $\thick T$ the smallest thick subcategory of $\mathscr{T}$ containing $T$.
\begin{itemize}
\item 
An object $T$ in $\mathscr{T}$ is called a \emph{presilting object} if it satisfies $\Hom_{\mathscr{T}}(T,T[i])=0$ for all $i>0$.
\item 
A presilting object $T$ in $\mathscr{T}$ is called a \emph{silting object} if it satisfies $\mathscr{T}=\thick T$.
\item 
An object $T$ in $\mathscr{T}$ is called a  \emph{pretilting object} if it satisfies $\Hom_{\mathscr{T}}(T,T[i])=0$ for all $i\neq0$.
\item 
A pretilting object $T$ in $\mathscr{T}$ is called a  \emph{tilting object} if it satisfies $\mathscr{T}=\thick T$.
\end{itemize}
\end{definition}

\begin{theorem}[{\cite{Keller:ddc}}]\label{Keller-tilting}
Let $\mathscr{T}$ be an algebraic triangulated category.
Assume that $\mathscr{T}$ is a Hom-finite $K$-linear Krull-Schmidt category.
If $T$ is a tilting object in $\mathscr{T}$, then there is an equivalence
\[
\mathscr{T} \cong \sfK^{\mrb}(\proj\End_{\mathscr{T}}(T))
\]
of triangulated categories.
If moreover, $\End_{\mathscr{T}}(T)$ has finite global dimension, then 
$\sfK^{\mrb}(\proj\End_{\mathscr{T}}(T))\cong \sfD^{\mrb}(\mod\End_{\mathscr{T}}(T))$ holds.
Consequently the above equivalence gives an equivalence 
\[
\mathscr{T} \cong \sfD^{\mrb}(\mod\End_{\mathscr{T}}(T))
\]
of triangulated categories.
\end{theorem}

The stable category of graded Cohen-Macaulay modules of a graded
IG-algebra is a typical example of a triangulated category which satisfies the assumption of Theorem \ref{Keller-tilting}.
Therefore, we can apply the above theorem to that categories. 

%
%
%
%
%
%
%

\subsection{Realizing stable categories as singularity categories}
We recall Buchweitz's equivalence which shows that the stable categories of Cohen-Macaulay modules can be realized as certain Verdier quotients of the derived categories.
For a graded algebra $A$, the Verdier quotient 
\[
\Sing^{\ZZ}A:=\sfD^{\mrb}(\grmod A)/\sfK^{\mrb}(\proj^{\ZZ}A)
\]
is called the \emph{singularity category} of $A$.
Buchweitz showed that the following equivalence exists if $A$ is IG.

\begin{theorem}[{\cite{Buchweitz}}]\label{Buchweitz-equiv}
Let $A=\bigoplus_{i=0}^{\ell}A_i$ be a graded IG-algebra.
Then, there is an equivalence $\beta:\stabgrCM A\to\Sing^{\ZZ} A$ of triangulated categories such that
the following diagram commutes.
\[
\xymatrix{
\grCM A \ar[r]^-{\mathrm{inc.}} \ar[d]_-{\mathrm{nat.}} & \grmod A \ar[r]^-{\mathrm{inc.}} & \sfD^{\mrb}(\grmod A) \ar[d]^-{\mathrm{quot.}} \\
\stabgrCM A \ar@{.>}[rr]_{\beta}^{\cong} & & \Sing^{\ZZ} A
}
\]
\end{theorem}

\subsection{Quasi-Veronese constructions}
In this subsection, we recall  basic facts about quasi-Veronese construction which play an important role in the study of representation theory of graded algebras. 
Let $A=\bigoplus_{i=0}^{\ell}A_i$ be a graded algebra with $A_{\ell}\neq0$. 
We recall that  the \emph{Beilinson algebra} $\nabla A$
and its bimodule $\Delta A$ are defined to be 
\begin{eqnarray}
\nabla A: = 
\begin{pmatrix} 
A_{0} & A_{1} & \cdots & A_{\ell -1} \\
0         & A_{0} & \cdots  & A_{\ell -2} \\
\vdots & \vdots     &        & \vdots \\
0         & 0   & \cdots  & A_{0}
\end{pmatrix}, \ \ \ 
\Delta A: = 
\begin{pmatrix} 
A_{\ell} & 0 & \cdots & 0 \\
A_{\ell- 1} & A_{\ell} & \cdots  &0 \\
\vdots & \vdots     &        & \vdots \\
A_{1} & A_{2}   & \cdots  & A_{\ell}
\end{pmatrix} 
\label{Beilinson}
\end{eqnarray}
where the algebra structure and the bimodule structure are 
given by the  matrix  multiplications. 
Then, 
we consider the trivial extension algebra 
\[
A^{[\ell]} := \nabla A \oplus \Delta A
\]
of $\nabla A$ by $\Delta A$.
We regard $A^{[\ell]}$ as a graded algebra with $(A^{[\ell]})_0=\nabla A$ and $(A^{[\ell]})_1= \Delta A$.
This graded algebra $A^{[\ell]}$ coincides with the \emph{$\ell$-th quasi-Veronese algebra} of $A$ introduced by Mori \cite[Definition 3.10]{Mori B-construction}.

By \cite[Lemma 3.12]{Mori B-construction}, 
$A$ and $A^{[\ell]} $ are graded Morita equivalent to each other. 
More precisely, 
the functor $\mathsf{qv}$ below gives a $K$-linear equivalence. 
\[
\begin{split}
 & \mathsf{qv}: \GrMod A \xrightarrow{ \ \cong  \ } \GrMod A^{ [ \ell ]}, \\
& \mathsf{qv}(M) := \bigoplus_{i \in \ZZ} \mathsf{qv}(M)_{i} \ \mbox{ where } \
\mathsf{qv}(M)_{i} =  M_{i\ell} \oplus M_{i \ell +1 } \oplus \cdots \oplus M_{( i+ 1)\ell -1 }.
\end{split}
\]
We will need the following observations.

\begin{lemma}\label{basics-quasi-V}
We keep the notations as above. 
The following assertions hold.
\begin{enumerate}[\rm (1)]
\item
The following diagram is commutative:
\[
\xymatrix{
\GrMod A \ar[r]_-{\mathsf{qv}}^-{\cong} \ar[d]_-{(\ell)} & \GrMod A^{[\ell]} \ar[d]^-{(1)} \\
\GrMod A \ar[r]_-{\mathsf{qv}}^-{\cong} & \GrMod A^{[\ell]}.
}
\]

\item 
The following equalities hold.
\begin{enumerate}[\rm (i)]
\item 
$\qv(\bigoplus_{r=0}^{\ell-1}A(-r))=A^{[\ell]}$.
\item 
$\qv(\bigoplus_{r=0}^{\ell-1} A/J_A(-r))=A^{[\ell]}/J_{A^{[\ell]}}$.
\item
$\qv(U)=\nabla A(1)$ and $\qv(V)=\Delta A$.
\end{enumerate}
Here $U=\bigoplus_{i=1}^{\ell}A(i)_{<0}$ as in \eqref{def-U}, and $V=\Omega_A(U)=\bigoplus_{i\geq1}A(i)_{\geq0}$ as in \eqref{def-T}.
\item
We have the following equivalences as restrictions of $\qv$.
\begin{enumerate}[\rm (i)]
\item 
$\qv:\grproj A \cong \grproj A^{[\ell]}$.
\item 
$\qv:\grinj A \cong \grinj A^{[\ell]}$.
\end{enumerate}

\item
$A$ is a graded $d$-IG-algebra if and only if  so is $A^{[\ell]}$.
If this is the case, we have the following equivalences from $\qv$.
\begin{enumerate}[\rm (i)]
\item
$\qv:\grCM A \cong \grCM A^{[\ell]}$.
\item
$\qv:\stabgrCM A \cong \stabgrCM A^{[\ell]}$.
\end{enumerate}
\end{enumerate}
\end{lemma}

%

The following lemma follows from \cite[III Proposition 2.7]{ARS}.

\begin{lemma}\label{lemma-ARS}
We keep the notations as above. 
Then, $A_0$ has finite global dimension if and only if so does $\nabla A$. 
\end{lemma}

We give an important remark. 

\begin{remark}\label{reduction-qv}
Assume that $A$ is a graded $1$-IG-algebra.
Then, $A^{[\ell]}$ is also a graded $1$-IG-algebra.
Let $V=\bigoplus_{i\geq 1}A(i)_{\geq0}$ and $V'=A^{[\ell]}(1)_{\geq 0}=\Delta A$. 
It follows from the above two lemmas that:
\begin{itemize}
\item 
$\gldim A_0<\infty$ $\Leftrightarrow$ $\gldim (A^{[\ell]})_0<\infty$.
\item
$\qv:\stabgrCM A \cong \stabgrCM A^{[\ell]}$.
\item
$V$ is a tilting object in $\stabgrCM A$ if and only if $V'$ is a tilting object in $\stabgrCM A^{[\ell]}$.
\end{itemize}
Therefore, the study of Question \ref{question-syzygy-U-tilting} is reduced to the case that $A$ is a trivial extension algebra.
\end{remark}

\subsection{Asid bimodules}

The second and third authors studied the representation theory of trivial extension IG-algebras \cite{adasore,higehaji}.
In this subsection, we recall their result which asserts that the singularity categories of trivial extension IG-algebras can be realized as full subcategories of the derived categories.

Let $\Lambda$ be an algebra, $C$ a $\Lambda$-$\Lambda$-bimodule,
and $A:=\Lambda \oplus C$ the trivial extension algebra of $\Lambda$ by $C$.
In this paper, we always regard the trivial extension algebra $A$ as a graded algebra with $A_0=\Lambda$ and $A_1=C$.

If $A$ is a graded IG-algebra, then we call $C$ an \emph{asid bimodule} over $\Lambda$.
In the study of trivial extension IG-algebras, the following numbers are important.



%
%

\begin{definition}\label{def-asid-number}
If $C$ is an asid bimodule, we call a non-negative integer
\[
\alpha_r(C):=\max \{ i \in\ZZ \ | \ \exists n\in\ZZ_{\geq0} \mbox{ s.t. }  \soc(\Omega^{-n}_A(A))_{-i}\neq0 \}+1
\]
the \emph{right asid number} of $C$ (see \cite[Definition 5.11 and Corollary 5.12]{adasore}).
Here $\Omega^{-n}_A(A)$ is the $n$-th graded cosyzygy of $A_A$. 
Dually, we can define the left asid number $\alpha_{\ell}(C)$ of $C$.

We remark that these two numbers coincide  if $\Lambda$ is of finite global dimension \cite[Theorem 4.12.(1)]{higehaji}.
In this case, we denote them by $\alpha(C):=\alpha_r(C)=\alpha_{\ell}(C)$, and call the asid number of $C$. 
\end{definition}


%
Now we regard $\Lambda$ as a graded algebra concentrated in degree $0$.
Then, a canonical map $A \to \Lambda$ is a homomorphism of graded algebras. 
This induces a functor $\sfD^{\mrb}(\mod\Lambda)\to\sfD^{\mrb}(\grmod A)$.
Let $\varpi$ be the composite of this functor and the quotient functor as follows:
\[
\varpi: \sfD^{\mrb}(\mod\Lambda)\to\sfD^{\mrb}(\grmod A) \xrightarrow{\mbox{\scriptsize quot.}}\Sing^{\ZZ}A.
\]
This functor gives an equivalence from some full subcategory $\sfT$ of $\sfD^{\mrb}(\mod\Lambda)$ to $\Sing^{\ZZ}A$ if $A$ is IG.
To describe the category $\sfT$, we need one notation.  
For a natural number $a$, we denote the iterated derived tensor product of $C$ by
\[
C^{a}:=C\lotimes_{\Lambda}C\lotimes_{\Lambda}\cdots \lotimes_{\Lambda}C \ \ \ (\mbox{$a$-times}).
\]
Also we set $C^0:=\Lambda$.

\begin{theorem}[{\cite[Theorem 1.4]{higehaji}}]\label{higehaji-equiv}
Assume that $\Lambda$ is of finite global dimension, and $C$ an asid bimodule over $\Lambda$.
Let $\alpha=\alpha(C)$ and $\sfT=\thick C^{\alpha}$.
Then, the functor
\[
\varpi|_{\sfT}:\sfT \to \Sing^{\ZZ}A.
\] 
is an equivalence of triangulated categories. 
Therefore, we also have an equivalence
\[
\mathcal{H}:\sfT \xrightarrow{\varpi|_{\sfT}} \Sing^{\ZZ}A \xrightarrow{\beta^{-1}} \stabgrCM A
\]
of triangulated categories. 
\end{theorem}

We summarize equivalences of Theorem \ref{Buchweitz-equiv} and Theorem \ref{higehaji-equiv} in the following figure. 
\[
\xymatrix{
\grCM A \ar[r]^-{\mathrm{inc.}} \ar[d]_-{\mathrm{nat.}} & \grmod A \ar[r]^-{\mathrm{inc.}} & \sfD^{\mrb}(\grmod A) \ar[d]^-{\mathrm{quot.}}  & \sfD^{\mrb}(\mod \Lambda) \ar[l]& \sfT=\thick C^{\alpha} \ar[l]_-{\ \mathrm{inc.}}  \ar@/^15pt/[lld]^-{\varpi|_{\sfT}}_{\cong} \\
\stabgrCM A \ar[rr]_{\beta}^{\cong} & & \Sing^{\ZZ} A 
}
\]

%
%

\subsection{The number $g(A)$.}

In this subsection, we introduce the number $g(M)$ for a graded module $M$ over a graded algebra, and study basic properties.
In particular, we give a relationship between $g(A)$ and asid numbers when $A$ is a trivial extension algebra. 

\begin{definition}\label{def-g(M)}
Let $A=\bigoplus_{i=0}^{\ell}A_i$ be a graded algebra.
For $M\in\grmod A$, we define an integer $g(M)$ by
\[
g(M):=\sup\{ i\in\ZZ \ | \  \exists n\in\ZZ_{\geq0} \mbox{ s.t. }  \Ext^n_{A}(A/J_A(i),M)_0\neq0 \}.
\]
We take a minimal graded injective resolution 
\[
0\to M\to I^0 \to I^1 \to \cdots \to I^n \to \cdots
\]
of $M$.
Then it is easy to observe that the following equalities hold.
\begin{eqnarray*}
g(M)
&=&\sup\{ i\in\ZZ \ | \ \exists n\in\ZZ_{\geq0} \mbox{ s.t. } \Hom_{A}(A/J_A(i),I^n)_0\neq0 \} \\
&=&\sup\{ i\in\ZZ \ | \ \exists n\in\ZZ_{\geq0} \mbox{ s.t. } \soc (I^n)_{-i} \neq0 \} \\ 
&=&\sup\{ i\in\ZZ \ | \ \exists n\in\ZZ_{\geq0} \mbox{ s.t. } \soc(\Omega_A^{-n}(M))_{-i} \neq0  \}. 
\end{eqnarray*}
Here $\Omega_A^{-n}(M)$ is the $n$-th graded cosyzygy of $M$.
\end{definition}


We give a relationship between $g(M)$ and $g(\qv(M))$.

\begin{lemma}\label{lem-general-asid}
Let $A=\bigoplus_{i=0}^{\ell}A_i$ be a graded algebra.
We consider the equivalence $\qv:\grmod A\to\grmod A^{[\ell]}$.
For $M,N\in\grmod A$, the following assertions hold.
\begin{enumerate}[\rm (1)]
\item
$g(M\oplus N)=\sup\{g(M),g(N)\}$.
\item 
$g(M)<\infty$ if and only if $g(\qv(M))<\infty$.
\item
Assume that $g(M)<\infty$.
Then, we have inequalities
\[
(g(\qv(M))-1) \ell<g(M)\leq g(\qv(M))\ell.
\]
\item
Assume that $A$ is a graded IG-algebra.
We put $B=A^{[\ell]}$.
Then, we have inequalities 
\[
( g(B)-1)\ell<g(A)\leq  g(B)\ell.
\]
\end{enumerate}
\end{lemma}

\begin{proof}
(1) It follows from the definition of $g$.

(2) (3)
By Lemma \ref{basics-quasi-V}, we have 
\[
\Ext^n_{A}(\bigoplus_{r=0}^{\ell-1} A/J_A(\ell i-r),M)_0
\cong
\Ext^n_{A^{[\ell]}}(A^{[\ell]}/J_{A^{[\ell]}}(i),\qv(M))_0.
\]
for  $i\in\ZZ$.
So, the assertion (2) follows.
Moreover if $g(M)<\infty$, then there is an integer $r \ (0\leq r <\ell)$ such that $g(M)=g(\qv(M))\ell-r$ by the above isomorphism.
Thus the assertion (3) follows.

(4)
Note that $g(\bigoplus_{r=0}^{\ell-1}A(-r))=g(A)$ holds by (1).
Therefore, by (3) and $\qv(\bigoplus_{r=0}^{\ell-1}A(-r))=A^{[\ell]}$, 
we have the desired inequalities.
\end{proof}

Let $g(A):=g(A_A)$ and $g(A^{\op}):=g({}_AA)$.
These numbers are related to asid numbers as in the following lemmas.

\begin{lemma}\label{asid=g}
Let $\Lambda$ be an algebra, and $C$ an asid bimodule over $\Lambda$.
We consider the trivial extension algebra $A=\Lambda\oplus C$.
Then, the following equality holds.
\begin{equation*}
\alpha_r(C)=g(A)+1.
\end{equation*}
\end{lemma}

\begin{proof}
The assertion follows from Definition \ref{def-asid-number} and the equalities mentioned in Definition \ref{def-g(M)}.
\end{proof}

\begin{lemma}\label{qv-g-asid}
Let $A=\bigoplus_{i=0}^{\ell}A_i$ be a graded IG-algebra.
We consider the Beilinson algebra $\nabla A$ and its bimodule $\Delta A$ as in \eqref{Beilinson}.
\begin{enumerate}[\rm(1)]
\item 
$\Delta A$ is an asid bimodule over $\nabla A$.
\item
We set $\alpha_r:=\alpha_r(\Delta A)$.
Then, we have inequalities 
\[
(\alpha_r-2)\ell  <  g(A)\leq (\alpha_r-1)\ell.
\]
In particular, 
\[
g(A)\leq0
\ \Leftrightarrow \
g(A^{[\ell]})\leq0  
\ \Leftrightarrow \
\alpha_r \leq 1.
\]
\end{enumerate}
\end{lemma}

\begin{proof}
(1)
As in Lemma \ref{basics-quasi-V} (4), $A^{[\ell]}$ is a graded IG-algebra, 
and so $\Delta A$ is an asid bimodule over $\nabla A$.

(2)
The assertion follows from Lemma \ref{lem-general-asid} and Lemma \ref{asid=g}.
\end{proof}

By the above lemma and the left right symmetry of asid numbers, we have the following corollary.

\begin{corollary}\label{symmetry g(A)<=0}
Let $A=\bigoplus_{i=0}^{\ell}A_i$ be a graded IG-algebra.
Assume that $A_0$ has finite global dimension.
Then, $g(A)\leq 0$ if and only if $g(A^{\op})\leq 0$.
\end{corollary}

\begin{proof}
We consider the Beilinson algebra $\nabla A$ of $A$. 
The bimodule $\Delta A$ is an asid bimodule over $\nabla A$ by Lemma \ref{qv-g-asid} (1).
We denote by $\alpha_r$ (resp. $\alpha_{\ell}$) the right (resp. left) asid number of $\Delta A$. 

By Lemma \ref{lemma-ARS}, $\nabla A$ has finite global dimension.
So we have $\alpha_r=\alpha_{\ell}$ by \cite[Theorem 4.12.(1)]{higehaji}.
Thus, by  Lemma \ref{qv-g-asid} (2) and its opposite version, it holds that
\[
g(A)\leq0
\ \Leftrightarrow \
\alpha_r=\alpha_{\ell} \leq 1.
\ \Leftrightarrow \
g(A^{\op})\leq0.
\qedhere
\]
\end{proof}

An upper bound of $g(A)$ is given by the injective dimension of $A_A$.

\begin{lemma}\label{g<id}
Let $A=\bigoplus_{i=0}^{\ell}A_i$ be a graded $d$-IG-algebra.
Then we have the inequality 
\[
-\ell \leq g(A)\leq d  \ell.
\]
\end{lemma}

\begin{proof}
Let 
\[
0 \to A \xrightarrow{f^0} I^{0} \xrightarrow{f^1} I^{1}  \xrightarrow{f^2} \cdots \xrightarrow{f^{d}}  I^{d}\to 0
\]
be a minimal injective resolution of $A_A$. 
Then, $\image f^m$ belongs to $\mod^{[-m\ell,\ell]}A$ for $0\leq m \leq d$.
Indeed, since  $A \in \mod^{[0,\ell]} A$ and $\tuD(A) \in \mod^{[-\ell, 0]} A$,
one can show the claim by induction on $m$.
By this fact, $\soc(I^m)$ belongs to $\mod^{[-m\ell,\ell]}A$ for $0\leq m \leq d$.
Thus we have $-\ell\leq g(A)\leq d\ell$.
\end{proof}

The following corollary follows from Lemma \ref{asid=g} and Lemma \ref{g<id}.

\begin{corollary}\label{bound of asid number}
Let $\Lambda$ be an algebra, $C$ an asid bimodule over $\Lambda$, 
and $A=\Lambda\oplus C$ the trivial extension algebra. 
We denote by $d$ the graded injective dimension of $A_A$.
Then, we have the inequalities
\[
0\leq \alpha_{r}(C) \leq d+1.
\]
\end{corollary}

\section{Proofs of Theorems \ref{Q(c)<->(d)}, \ref{silting theorem} and \ref{tilting theorem 1}}

\subsection{Proof of Theorem \ref{Q(c)<->(d)}}

In this subsection, we give a proof of Theorem \ref{Q(c)<->(d)}.
To prove it, we use the dimension $\tridim \mathscr{T}$ of a triangulated category $\mathscr{T}$ introduced by Rouquier. For the definition, readers should refer to \cite{Rouquier}.

\begin{proof}[Proof of Theorem \ref{Q(c)<->(d)}]
We assume that $\stabgrCM A$ has a tilting object, whose endomorphism algebra is denoted by $\Gamma$.
First we claim $\tridim \stabgrCM A<\infty$. 
We have $\gldim A_0<\infty$ by \cite[Lemma 3.0.1]{Lu-Zhu} (see also \cite[Proposition 6.10]{cotilting type}), and so $\gldim \nabla A<\infty$ by Lemma \ref{lemma-ARS}. 
Therefore, we have $\tridim \sfD^{\mrb}(\mod \nabla A) < \infty$ by \cite[Proposition 7.25]{Rouquier}.  
By Lemma \ref{basics-quasi-V} (4) and Theorem \ref{higehaji-equiv}, 
there is a dense functor $\sfD^{\mrb}(\mod \nabla A)\to\stabgrCM A$ of triangulated categories.
Thus, it follows from \cite[Lemma 3.4]{Rouquier} that $\tridim \stabgrCM A \leq \tridim \sfD^{\mrb}(\mod \nabla A) < \infty$. 
The claim holds.

By our assumption and Theorem \ref{Keller-tilting}, 
there is an equivalence $\stabgrCM A \cong \sfK^{\mrb}(\proj \Gamma)$ of triangulated categories.
So we obtain the inequality $\tridim \sfK^{\mrb}(\proj \Gamma)=\tridim \stabgrCM A < \infty$ by the above claim. 
Thus, we conclude $\gldim \Gamma < \infty$  by \cite[Proposition 7.25]{Rouquier}.  
\end{proof}

%
%

\subsection{Proof of Theorem \ref{silting theorem}}

In this subsection, we give a proof of Theorem \ref{silting theorem} and a remark about it.



\begin{proof}[Proof of Theorem \ref{silting theorem}]
(1)
We show that $\stabHom_{A}(V,V[i])_0\cong\Ext_{A}^{i}(V,V)_0=0$ for all $i>0$.
Let $Q:=\bigoplus_{i=1}^{\ell}A(i)$.
We claim that $\Hom_A(M,V)_0\cong\Hom_A(M,Q)_0$ holds for any $M\in\mod^{[0,\infty)}A$.
Indeed, there is an exact sequence
\begin{eqnarray}
0\to V \to Q \to U\to0 \label{exact-T-U}
\end{eqnarray}
For any $M\in\mod^{[0,\infty)}A$, we have $\Hom_A(M,U)_0=0$ since $U\in\mod^{(-\infty,0)}A$.
So applying $\Hom_A(M,-)_0$ to the above exact sequence, we have the desired isomorphism.

Now we take a minimal projective resolution $\cdots\to P^n \to \cdots \to P^1 \to P^0 \to V \to0$ of $V$ in $\grmod A$.
Note that each term belongs to $\mod^{[0,\infty)}A$.
So by the claim, there is the following isomorphism of complexes.
\[
\xymatrix@=15pt{
\Hom_A(P^0,V)_0 \ar[d]_{\cong} \ar[r] & \Hom_A(P^1,V)_0 \ar[d]_{\cong} \ar[r] & \cdots  \ar[r] &  \Hom_A(P^n,V)_0 \ar[d]_{\cong} \ar[r] & \cdots \\
\Hom_A(P^0,Q)_0 \ar[r] & \Hom_A(P^1,Q)_0 \ar[r] & \cdots \ar[r]  &  \Hom_A(P^n,Q)_0 \ar[r]  &  \cdots
}
\]
By this isomorphism and since $V$ is Cohen-Macaulay, $\Ext_{A}^{i}(V,V)_0=\Ext_{A}^{i}(V,Q)_0=0$ holds.

(2)
By (1), the remaining part is to show $\stabgrCM A=\thick V$.
Since the exact sequence \eqref{exact-T-U} gives rise to an triangle $V\to Q\to U \to V[1]$ in $\sfD^{\mrb}(\grmod A)$, we have $\beta(V)\cong U[-1]$ in $\Sing^{\ZZ}A$ (see Theorem \ref{Buchweitz-equiv}).
Therefore, it is enough to show that $\Sing^{\ZZ}A=\thick U$.

Let $\mathcal{U}$ be the full subcategory of $\grmod A$ satisfying the following conditions.
\begin{itemize}
\item
$\mathcal{U}$ contains $U$, graded projective $A$-modules and graded injective $A$-modules.
\item 
$\mathcal{U}$ is closed under taking direct sums, direct summands and extensions.
\item
$\mathcal{U}$ is closed under taking kernels of epimorphisms and cokernels of monomorphisms.
\end{itemize}
As in the proof of \cite[Lemma 3.5]{Yamaura}, 
the equality $\mathcal{U}=\grmod A$ holds.
Thus, we have $\Sing^{\ZZ}A=\thick\mathcal{U}\subset \thick U$, and so $\Sing^{\ZZ}A=\thick U$.
\end{proof}

By Theorem \ref{silting theorem}, it is natural to ask whether or not the $d$-th syzygy $\Omega^{d}_A(U)$ of $U$ is a presilting object in $\stabgrCM A$ for a graded $d$-IG-algebra $A$ with $d\geq2$.
The following example tells us that it is not a presilting object in general.

\begin{example}\label{counter-silting}
Let $A$ be an algebra defined by a quiver
\[
\xymatrix{
1 \ar@/^5pt/[r]^{\beta} \ar@(ul,dl)_{\alpha} & 2 \ar@/^5pt/[l]^{\gamma}
}
\]
with relations $\alpha^2=\beta\gamma$ and $\alpha\beta=\gamma\beta=\gamma\alpha=0$.
We introduce a structure of graded algebra of $A$ by $\deg\alpha=1$, $\deg\beta=0$, $\deg\gamma=2$.
Then, $A$ is a graded 2-IG-algebra.

Let $e_i$ be the idempotent of $A$ corresponding to the vertex $i$.
The Loewy series of indecomposable projective $A$-modules $e_1A,e_2A$ and indecomposable injective $A$-modules $\tuD(Ae_1),\tuD(Ae_2)$ are as follows.
\[
	\begin{xy} 
	(-15,7)*{e_1A=},
	(0,0)="O",
	"O"+<0cm,1.2cm>="11"*{{\bf 1}},
	"11"+<-0.3cm,-0.5cm>="21"*{1},
	"21"+<-0.3cm,-0.5cm>="31"*{},
	"21"+/r0.6cm/="22"*{{\bf2}},
	"31"+/r0.6cm/="32"*{1},
	"32"+/r0.6cm/="33"*{},
	\ar@{.}"11"+/dl/;"21"+/u/<2pt>
	\ar@{-}"11"+/dr/;"22"+/u/<-2pt>
	\ar@{.}"21"+/dr/;"32"+/u/<-2pt>
	\ar@{:}"22"+/dl/;"32"+/u/<2pt>
	\end{xy},
	\hspace{10mm}
	\begin{xy} 
	(-15,7)*{e_2A=},
	(0,0)="O",
	"O"+<0cm,1.2cm>="11"*{{\bf 2}},
	"11"+<-0.3cm,-0.5cm>="21"*{},
	"21"+<-0.3cm,-0.5cm>="31"*{},
	"21"+/r0.6cm/="22"*{},
	"31"+/r0.6cm/="32"*{1},
	"32"+/r0.6cm/="33"*{},
	\ar@{:}"11"+/d/;"32"+/u/
	\end{xy},
	\hspace{10mm}
	\begin{xy} 
	(-15,7)*{\tuD(Ae_1)=},
	(0,0)="O",
	"O"+<0cm,1.2cm>="11"*{1},
	"11"+<-0.3cm,-0.5cm>="21"*{1},
	"21"+<-0.3cm,-0.5cm>="31"*{},
	"21"+/r0.6cm/="22"*{2},
	"31"+/r0.6cm/="32"*{{\bf 1}},
	"32"+/r0.6cm/="33"*{},
	\ar@{.}"11"+/dl/;"21"+/u/<2pt>
	\ar@{-}"11"+/dr/;"22"+/u/<-2pt>
	\ar@{.}"21"+/dr/;"32"+/u/<-2pt>
	\ar@{:}"22"+/dl/;"32"+/u/<2pt>
	\end{xy},
	\hspace{10mm}
	\begin{xy} 
	(-15,7)*{\tuD(Ae_2)=},
	(0,0)="O",
	"O"+<0cm,1.2cm>="11"*{{\bf 1}},
	"11"+<-0.3cm,-0.5cm>="21"*{},
	"21"+<-0.3cm,-0.5cm>="31"*{},
	"21"+/r0.6cm/="22"*{{\bf 2}},
	"31"+/r0.6cm/="32"*{},
	"32"+/r0.6cm/="33"*{},
	\ar@{-}"11"+/dr/;"22"+/u/<-2pt>
	\end{xy}.
\]
The numbers written in boldface are in degree $0$,
and numbers connected by solid lines (resp. dotted lines, double dotted lines) are simple modules connected by arrows in degree $0$ (resp. degree $1$, degree $2$).

In this example, $\Omega_A^2(U)$ is not a presilting object in $\stabgrCM A$.
Indeed,  $\Omega_A^2(U)$ is given by $V=e_2A \oplus X(-1) \oplus X^{\oplus 2}$ where $X$ is as follows.
\[
	\begin{xy} 
	(-15,7)*{X=},
	(0,0)="O",
	"O"+<0cm,1.2cm>="11"*{},
	"11"+<-0.3cm,-0.5cm>="21"*{1},
	"21"+<-0.3cm,-0.5cm>="31"*{},
	"21"+/r0.6cm/="22"*{{\bf 2}},
	"31"+/r0.6cm/="32"*{1},
	"32"+/r0.6cm/="33"*{},
	\ar@{.}"21"+/dr/;"32"+/u/<-2pt>
	\ar@{:}"22"+/dl/;"32"+/u/<2pt>
	\end{xy}
\]
By $X[1]=X(1)$, we have 
\[
\stabHom_A(\Omega_A^2(U),\Omega_A^2(U)[1])_0=\stabHom_A(X(-1)\oplus X^{\oplus2},X\oplus X^{\oplus2}(1))_0\neq0.
\]
Thus, $\Omega_A^2(U)$ is not a presilting object in $\stabgrCM A$.

\end{example}

\subsection{Proof of Theorem \ref{tilting theorem 1}}

We give a proof of Theorem \ref{tilting theorem 1}.
First we prove it for the case of trivial extension algebras.

\begin{lemma}\label{tilting theorem 1 trivial extension}
Let $\Lambda$ be an algebra of finite global dimension, and $C$ an asid bimodule over $\Lambda$.
Assume that the trivial extension algebra $A=\Lambda\oplus C$ is a graded $1$-IG-algebra.
If  $\alpha(C)\leq 1$ is satisfied, then $C$ is a tilting object in $\stabgrCM A$.
\end{lemma}

\begin{proof}
To avoid confusion, if we regard $C$ as a $\Lambda$-module, we write it by $C_{\Lambda}$.
On the other hand, if we regard $C$ as a graded $A$-module concentrated in degree $0$, we write it by $C_A$.

By Theorem \ref{silting theorem}, $C_A$ is a silting object in $\stabgrCM A$.
So it is enough to show that $\stabHom_{A}(C_A,C_A[i])_0=0$ for all $i<0$.
By Theorem \ref{higehaji-equiv}, we have equivalences
\[
\xymatrix{
\sfT \ar[r]^-{\cong}_-{\varpi|_{\sfT}} & \Sing^{\ZZ}A & \stabgrCM A. \ar[l]_-{\cong}^-{\beta}
}
\]
Note that $\sfT=\sfD^{\mrb}(\mod\Lambda)$ if $\alpha(C)=0$, and $\sfT=\thick C_{\Lambda}$ if $\alpha(C)=1$.
Therefore, $\sfT$ contains $C_{\Lambda}$ in both cases. 
Moreover, since $\varpi(C_{\Lambda})=C_A=\beta(C_A)$ hold in $\Sing^{\ZZ}A$,  we have $\mathcal{H}(C_{\Lambda})\cong C_A$ in $\stabgrCM A$.
Thus, we have
\[
\stabHom_{A}(C_{A},C_{A}[i])_0 \cong \Hom_{\sfD^{\mrb}(\mod \Lambda)}(C_{\Lambda}, C_{\Lambda}[i]) = 0
\]
for all $i<0$.
\end{proof}

By Lemma \ref{tilting theorem 1 trivial extension} and the quasi-Veronese algebra construction, we can prove Theorem \ref{tilting theorem 1}.

\begin{proof}[Proof of Theorem \ref{tilting theorem 1}]
We take the $\ell$-th quasi-Veronese algebra $A^{[\ell]}=\nabla A\oplus \Delta A$ of $A$. 
Since $A_0$ is of finite global dimension, so is $\nabla A$ by Lemma \ref{lemma-ARS}.
Moreover by $g(A)\leq 0$ and Lemma \ref{qv-g-asid} (2), we have $\alpha(\Delta A)\leq 1$.
Therefore, $\Delta A$ is a tilting object in $\stabgrCM A^{[\ell]}$ by Lemma \ref{tilting theorem 1 trivial extension}.
Thus, $V$ is a tilting object in $\stabgrCM A$ by Remark \ref{reduction-qv}. 
\end{proof}

We remark that for a graded $1$-IG-algebra $A$ with $g(A)=1$, 
the finiteness of global dimension of $A_0$ does not imply that $V$ is a tilting object in $\stabgrCM A$.
Actually, we have two graded $1$-IG-algebras $A$ with $g(A)=1$ in Example \ref{example-g(A)=1}.
For the first one, $V$ is a tilting object in $\stabgrCM A$.
For the other, $V$ is not a tilting object in $\stabgrCM A$.

\begin{example}\label{example-g(A)=1}
As examples, we deal with the truncated preprojective algebras $\Pi(Q)_w$ where $Q$ is 
\[
\begin{xy}
	(0,10)="O",
	"O"+<0cm,0cm>="1"*{1},
	"1"+<-0.8cm,-1.2cm>="2"*{2},
	"1"+<0.8cm,-1.2cm>="3"*{ \ 3.},
	\ar "1"+<-0.1cm,-0.2cm>;"2"+<0.2cm,0.2cm>_{\alpha}
	\ar "1"+<0.1cm,-0.2cm>;"3"+<-0.2cm,0.2cm>^{\gamma}
	\ar "2"+<0.2cm,-0cm>;"3"+<-0.2cm,0cm>_{\beta}
\end{xy}
\]
and $w$ is an element in the Coxeter group $W_Q$ associated to $Q$.
For the definition and notations, we refer to Section \ref{section-trucated-pp-alg}.
In stead of introducing them, we give $\Pi(Q)_w$ by quivers with relations here.
\begin{enumerate}[(1)]
\item 
Let $A$ be a graded algebra defined by the following graded quiver with relations.
\[
\begin{xy}
	(0,10)="O",
	"O"+<0cm,0cm>="1"*{1},
	"1"+<-1cm,-1.5cm>="2"*{2},
	"1"+<1cm,-1.5cm>="3"*{3},
	\ar@/^/ "2"+<0cm,0.2cm>;"1"+<-0.3cm,-0.1cm>^{\alpha^*}
	\ar "1"+<0.1cm,-0.2cm>;"3"+<-0.2cm,0.2cm>_{\gamma}
	\ar@/_/ "3"+<0cm,0.2cm>;"1"+<0.3cm,-0.1cm>_{\gamma^*}
	\ar "2"+<0.2cm,-0cm>;"3"+<-0.2cm,0cm>_{\beta}
\end{xy} 
\hspace{5mm}
\begin{cases}
\deg \beta=\deg\gamma=0 \\
\deg \alpha^*=\deg\gamma^*=1,
\end{cases}
\mbox{ with } 
\hspace{5mm}
\begin{cases}
\gamma\gamma^*=0 \\
\gamma^*\gamma=0
\end{cases}
\]
This algebra is the truncated preprojective algebra $\Pi(Q)_w$ associated to $w=s_2s_3s_1s_3$.
The Loewy series of indecomposable projective $A$-modules $e_1A,e_2A,e_3A$ and indecomposable injective $A$-modules $\tuD(Ae_1),\tuD(Ae_2),\tuD(Ae_3)$ are as follows.
\[
	\begin{xy} 
	(-15,7)*{e_1A=},
	(0,0)="O",
	"O"+<0cm,1.2cm>="11"*{{\bf 1}},
	"11"+<-0.3cm,-0.5cm>="21"*{},
	"21"+<-0.3cm,-0.5cm>="31"*{},
	"21"+/r0.6cm/="22"*{{\bf 3}},
	"31"+/r0.6cm/="32"*{},
	"32"+/r0.6cm/="33"*{},
	\ar@{-}"11"+/dr/;"22"+/u/<-2pt>
	\end{xy},
\hspace{10mm}
	\begin{xy} 
	(-15,7)*{e_2A=},
	(0,0)="O",
	"O"+<0cm,1.2cm>="11"*{{\bf 2}},
	"11"+<-0.3cm,-0.5cm>="21"*{{\bf 3}},
	"21"+<-0.3cm,-0.5cm>="31"*{1},
	"21"+/r0.6cm/="22"*{1},
	"31"+/r0.6cm/="32"*{},
	"32"+/r0.6cm/="33"*{3},

	\ar@{-}"11"+/dl/;"21"+/u/<2pt>
	\ar@{.}"11"+/dr/;"22"+/u/<-2pt>
	\ar@{.}"21"+/dl/;"31"+/u/<2pt>
	\ar@{-}"22"+/dr/;"33"+/u/<-2pt>
	\end{xy},
\hspace{10mm}
	\begin{xy} 
	(-15,7)*{e_3A=},
	(0,0)="O",
	"O"+<0cm,1.2cm>="11"*{{\bf 3}},
	"11"+<-0.3cm,-0.5cm>="21"*{1},
	"21"+<-0.3cm,-0.5cm>="31"*{},
	"21"+/r0.6cm/="22"*{},
	"31"+/r0.6cm/="32"*{},
	"32"+/r0.6cm/="33"*{},

	\ar@{.}"11"+/dl/;"21"+/u/<2pt>
	\end{xy}.
\]
\ 
\[
	\begin{xy}
	(-20,17)*{I_1=\tuD(Ae_1)=},
	(0,0)="O",
	"O"+<0cm,1.2cm>="11"*{{\bf 1}},
	"11"+<-0.3cm,+0.5cm>="21"*{2},
	"21"+<-0.3cm,+0.5cm>="31"*{},
	"21"+/r0.6cm/="22"*{3},
	"31"+/r0.6cm/="32"*{},
	"32"+/r0.6cm/="33"*{2},
	
	\ar@{.}"21"+/d/;"11"+/lu/<2pt>
	\ar@{.}"22"+/d/;"11"+/ru/<-2pt>
	\ar@{-}"33"+/d/;"22"+/ru/<-2pt>
	\end{xy},
\hspace{10mm}
	\begin{xy} 
	(-17,17)*{I_2=\tuD(Ae_2)=},
	(0,0)="O",
	"O"+<0cm,1.2cm>="11"*{{\bf 2}},
	"11"+<-0.3cm,+0.5cm>="21"*{},
	"21"+<-0.3cm,+0.5cm>="31"*{},
	"21"+/r0.6cm/="22"*{},
	"31"+/r0.6cm/="32"*{},
	"32"+/r0.6cm/="33"*{},
	\end{xy},
\hspace{10mm}
	\begin{xy} 
	(-23,17)*{I_3=\tuD(Ae_3)=},
	(0,0)="O",
	"O"+<0cm,1.2cm>="11"*{{\bf 3}},
	"11"+<-0.3cm,+0.5cm>="21"*{{\bf 1}},
	"21"+<-0.3cm,+0.5cm>="31"*{2},
	"21"+/r0.6cm/="22"*{{\bf 2}},
	"31"+/r0.6cm/="32"*{},
	"32"+/r0.6cm/="33"*{},

	\ar@{-}"21"+/d/;"11"+/lu/<2pt>
	\ar@{-}"22"+/d/;"11"+/ru/<-2pt>
	\ar@{.}"31"+/d/;"21"+/lu/<2pt>
	\end{xy}.
\]
This algebra $A$ is a graded $1$-IG-algebra with $g(A)=1$. 
In fact, minimal graded injective resolutions of indecomposable projective $A$-modules are given as follows. 
\[
0 \to e_1A \to I_3 \to I_2(1) \oplus I_2 \to 0,
\]
\[
0 \to e_2A \to I_1(-1)\oplus I_3(-1)  \to I_2^{\oplus 2} \oplus I_2(-1) \to 0,
\]
\[
0 \to e_3A \to I_1(-1) \to I_2^{\oplus 2} \to 0.
\]

In this example, $V$ is a tilting object in $\stabgrCM A$.
Indeed, $V$ is given by $V=X^{\oplus2}\oplus e_1A$ where $X$ is as follows.
\[
	\begin{xy} 
	(-8,10)*{X=},
	(0,0)="O",
	"O"+<0cm,1.2cm>="11"*{{\bf 1}},
	"11"+<-0.3cm,-0.5cm>="21"*{},
	"21"+/r0.6cm/="22"*{},
	\end{xy}.
\]
By Theorem \ref{silting theorem} (2), $V$ is a silting object in $\stabgrCM A$.
Moreover, it is easy to check that $\underline{\Hom}_A(V,V[i])_0=0$ for all $i<0$.
Thus, $V$ is a tilting object in $\stabgrCM A$.
Since $\underline{\End}_A(T)_0$ is Morita equivalent to $K$, there is an equivalence $\stabgrCM A \cong \sfD^{\mrb}(\mod K)$ of triangulated categories.

\item
Let $A$ be a graded algebra defined by the following graded quiver with relations.
\[
\begin{xy}
	(0,10)="O",
	"O"+<0cm,0cm>="1"*{1},
	"1"+<-1cm,-1.5cm>="2"*{2},
	"1"+<1cm,-1.5cm>="3"*{3},
	\ar "1"+<-0.1cm,-0.2cm>;"2"+<0.2cm,0.2cm>^{\alpha}
	\ar@/^/ "2"+<0cm,0.2cm>;"1"+<-0.3cm,-0.1cm>^{\alpha^*}
	\ar "1"+<0.1cm,-0.2cm>;"3"+<-0.2cm,0.2cm>_{\gamma}
	\ar@/_/ "3"+<0cm,0.2cm>;"1"+<0.3cm,-0.1cm>_{\gamma^*}
	\ar "2"+<0.2cm,-0cm>;"3"+<-0.2cm,0cm>^{\beta}
	\ar@/^/ "3"+<-0.2cm,-0.2cm>;"2"+<0.2cm,-0.2cm>^{\beta^*}
\end{xy} 
\hspace{5mm}
\begin{cases}
\deg \alpha=\deg \beta=\deg\gamma=0 \\
\deg \alpha^*=\deg \beta^*=\deg\gamma^*=1,
\end{cases}
\mbox{ with } 
\hspace{5mm}
\begin{cases}
\alpha\alpha^*+\gamma\gamma^*=0 \\
\alpha^*\alpha-\beta\beta^*=0 \\
\gamma^*\gamma=\beta^*\beta=0 \\
\gamma^*\alpha\beta=0\\
\beta^*\alpha^*\gamma=0 \\
\alpha\beta=0
\end{cases}
\]
This algebra is the truncated preprojective algebra $\Pi(Q)_w$ associated to $w=s_2s_3s_1s_3s_2s_1$.
The Loewy series of indecomposable projective $A$-modules $e_1A,e_2A,e_3A$ and indecoposable injective $A$-modules $\tuD(Ae_1),\tuD(Ae_2),\tuD(Ae_3)$ are as follows.
\[
	\begin{xy} 
	(-15,2)*{e_1A=},
	(0,0)="O",
	"O"+<0cm,1.2cm>="11"*{{\bf 1}},
	"11"+<-0.3cm,-0.5cm>="21"*{{\bf 2}},
	"21"+<-0.3cm,-0.5cm>="31"*{},
	"31"+<-0.3cm,-0.5cm>="41"*{},
	"41"+<-0.3cm,-0.5cm>="51"*{},
	"21"+/r0.6cm/="22"*{{\bf 3}},
	"31"+/r0.6cm/="32"*{1},
	"32"+/r0.6cm/="33"*{2},
	"41"+/r0.6cm/="42"*{},
	"42"+/r0.6cm/="43"*{},
	"43"+/r0.6cm/="44"*{1},
	"51"+/r0.6cm/="52"*{},
	"52"+/r0.6cm/="53"*{},
	"53"+/r0.6cm/="54"*{},
	"54"+/r0.6cm/="55"*{},
	
	\ar@{-}"11"+/dl/;"21"+/u/<2pt>
	\ar@{-}"11"+/dr/;"22"+/u/<-2pt>
	\ar@{.}"21"+/dr/;"32"+/u/<-2pt>
	\ar@{.}"22"+/dl/;"32"+/u/<2pt>
	\ar@{.}"22"+/dr/;"33"+/u/<-2pt>
	\ar@{.}"33"+/dr/;"44"+/u/<-2pt>
	\end{xy},
\hspace{10mm}
	\begin{xy} 
	(-17,2)*{e_2A=},
	(0,0)="O",
	"O"+<0cm,1.2cm>="11"*{{\bf 2}},
	"11"+<-0.3cm,-0.5cm>="21"*{{\bf 3}},
	"21"+<-0.3cm,-0.5cm>="31"*{1},
	"31"+<-0.3cm,-0.5cm>="41"*{2},
	"41"+<-0.3cm,-0.5cm>="51"*{},
	"21"+/r0.6cm/="22"*{1},
	"31"+/r0.6cm/="32"*{2},
	"32"+/r0.6cm/="33"*{3},
	"41"+/r0.6cm/="42"*{},
	"42"+/r0.6cm/="43"*{1},
	"43"+/r0.6cm/="44"*{2},
	"51"+/r0.6cm/="52"*{},
	"52"+/r0.6cm/="53"*{},
	"53"+/r0.6cm/="54"*{},
	"54"+/r0.6cm/="55"*{1},

	\ar@{-}"11"+/dl/;"21"+/u/<2pt>
	\ar@{.}"11"+/dr/;"22"+/u/<-2pt>
	\ar@{.}"21"+/dl/;"31"+/u/<2pt>
	\ar@{.}"21"+/dr/;"32"+/u/<-2pt>
	\ar@{-}"22"+/dl/;"32"+/u/<2pt>
	\ar@{-}"22"+/dr/;"33"+/u/<-2pt>
	\ar@{-}"31"+/dl/;"41"+/u/<2pt>
	\ar@{.}"32"+/dr/;"43"+/u/<-2pt>
	\ar@{.}"33"+/dl/;"43"+/u/<2pt>
	\ar@{.}"33"+/dr/;"44"+/u/<-2pt>
	\ar@{.}"44"+/dr/;"55"+/u/<-2pt>
	\end{xy},
\hspace{10mm}
	\begin{xy} 
	(-17,2)*{e_3A=},
	(0,0)="O",
	"O"+<0cm,1.2cm>="11"*{{\bf 3}},
	"11"+<-0.3cm,-0.5cm>="21"*{1},
	"21"+<-0.3cm,-0.5cm>="31"*{2},
	"31"+<-0.3cm,-0.5cm>="41"*{},
	"41"+<-0.3cm,-0.5cm>="51"*{},
	"21"+/r0.6cm/="22"*{2},
	"31"+/r0.6cm/="32"*{},
	"32"+/r0.6cm/="33"*{1},
	"41"+/r0.6cm/="42"*{},
	"42"+/r0.6cm/="43"*{},
	"43"+/r0.6cm/="44"*{},
	"51"+/r0.6cm/="52"*{},
	"52"+/r0.6cm/="53"*{},
	"53"+/r0.6cm/="54"*{},
	"54"+/r0.6cm/="55"*{},
	
	\ar@{.}"11"+/dl/;"21"+/u/<2pt>
	\ar@{.}"11"+/dr/;"22"+/u/<-2pt>
	\ar@{-}"21"+/dl/;"31"+/u/<2pt>
	\ar@{.}"22"+/dr/;"33"+/u/<-2pt>
	\end{xy}.
\]
\[
	\begin{xy} 
	(-30,22)*{I_1=\tuD(Ae_1)=},
	(0,0)="O",
	"O"+<0cm,1.2cm>="11"*{{\bf 1}},
	"11"+<-0.3cm,+0.5cm>="21"*{2},
	"21"+<-0.3cm,+0.5cm>="31"*{3},
	"31"+<-0.3cm,+0.5cm>="41"*{1},
	"41"+<-0.3cm,+0.5cm>="51"*{2},
	"21"+/r0.6cm/="22"*{3},
	"31"+/r0.6cm/="32"*{1},
	"32"+/r0.6cm/="33"*{2},
	"41"+/r0.6cm/="42"*{2},
	"42"+/r0.6cm/="43"*{},
	"43"+/r0.6cm/="44"*{},
	"51"+/r0.6cm/="52"*{},
	"52"+/r0.6cm/="53"*{},
	"53"+/r0.6cm/="54"*{},
	"54"+/r0.6cm/="55"*{},

	\ar@{.}"21"+/d/;"11"+/lu/<2pt>
	\ar@{.}"22"+/d/;"11"+/ru/<-2pt>
	\ar@{.}"31"+/d/;"21"+/lu/<2pt>
	\ar@{-}"32"+/d/;"21"+/ru/<-2pt>	
	\ar@{-}"32"+/d/;"22"+/lu/<2pt>
	\ar@{-}"33"+/d/;"22"+/ru/<-2pt>
	\ar@{-}"41"+/d/;"31"+/lu/<2pt>
	\ar@{-}"42"+/d/;"31"+/ru/<-2pt>
	\ar@{.}"42"+/d/;"32"+/lu/<2pt>
	\ar@{.}"51"+/d/;"41"+/lu/<2pt>
	\end{xy},
\hspace{5mm}
	\begin{xy}
	(-25,22)*{I_2=\tuD(Ae_2)=},
	(0,0)="O",
	"O"+<0cm,1.2cm>="11"*{{\bf 2}},
	"11"+<-0.3cm,+0.5cm>="21"*{3},
	"21"+<-0.3cm,+0.5cm>="31"*{1},
	"31"+<-0.3cm,+0.5cm>="41"*{2},
	"41"+<-0.3cm,+0.5cm>="51"*{},
	"21"+/r0.6cm/="22"*{{\bf 1}},
	"31"+/r0.6cm/="32"*{2},
	"32"+/r0.6cm/="33"*{3},
	"41"+/r0.6cm/="42"*{},
	"42"+/r0.6cm/="43"*{},
	"43"+/r0.6cm/="44"*{2},
	"51"+/r0.6cm/="52"*{},
	"52"+/r0.6cm/="53"*{},
	"53"+/r0.6cm/="54"*{},
	"54"+/r0.6cm/="55"*{},

	\ar@{.}"21"+/d/;"11"+/lu/<2pt>
	\ar@{-}"22"+/d/;"11"+/ru/<-2pt>
	\ar@{-}"31"+/d/;"21"+/lu/<2pt>
	\ar@{-}"32"+/d/;"21"+/ru/<-2pt>	
	\ar@{.}"32"+/d/;"22"+/lu/<2pt>
	\ar@{.}"33"+/d/;"22"+/ru/<-2pt>
	\ar@{.}"41"+/d/;"31"+/lu/<2pt>
	\ar@{-}"44"+/d/;"33"+/ru/<-2pt>
	\end{xy},
\hspace{5mm}
	\begin{xy} 
	(-25,22)*{I_3=\tuD(Ae_3)=},
	(0,0)="O",
	"O"+<0cm,1.2cm>="11"*{{\bf 3}},
	"11"+<-0.3cm,+0.5cm>="21"*{{\bf 1}},
	"21"+<-0.3cm,+0.5cm>="31"*{2},
	"31"+<-0.3cm,+0.5cm>="41"*{},
	"41"+<-0.3cm,+0.5cm>="51"*{},
	"21"+/r0.6cm/="22"*{{\bf 2}},
	"31"+/r0.6cm/="32"*{},
	"32"+/r0.6cm/="33"*{},
	"41"+/r0.6cm/="42"*{},
	"42"+/r0.6cm/="43"*{},
	"43"+/r0.6cm/="44"*{},
	"51"+/r0.6cm/="52"*{},
	"52"+/r0.6cm/="53"*{},
	"53"+/r0.6cm/="54"*{},
	"54"+/r0.6cm/="55"*{},

	\ar@{-}"21"+/d/;"11"+/lu/<2pt>
	\ar@{-}"22"+/d/;"11"+/ru/<-2pt>
	\ar@{.}"31"+/d/;"21"+/lu/<2pt>
	\end{xy}.
\]
This algebra $A$ is a graded $1$-IG-algebra with $g(A)=1$. 
In fact, minimal graded injective resolutions of indecomposable projective $A$-modules are given as follows. 
\[
0 \to e_1A \to I_1(-1) \oplus  I_1(-2) \to I_3(1)\oplus I_3\oplus I_3(-1) \to 0,
\]
\[
0 \to e_2A \to I_2(-1) \oplus  I_1(-2) \oplus  I_1(-3)  \to I_3^{\oplus 2} \oplus I_3(-1)\oplus I_3(-2) \to 0,
\]
\[
0 \to e_3A \to I_2(-1) \oplus  I_1(-2) \to I_3^{\oplus 2} \oplus I_3(-1) \to 0.
\]

In this example, $V$ is not a tilting object in $\stabgrCM A$.
Indeed, $V$ is given by $V=X^{\oplus5}\oplus Y^{\oplus2} \oplus Z^{\oplus3} \oplus e_1A$ where $X,Y,Z$ are as follows.
\[
	\begin{xy} 
	(-13,7)*{X=},
	(0,0)="O",
	"O"+<0cm,1.2cm>="11"*{{\bf 1}},
	"11"+<-0.3cm,-0.5cm>="21"*{},
	"21"+/r0.6cm/="22"*{},
	\end{xy},
	\hspace{15mm}
	\begin{xy} 
	(-13,7)*{Y=},
	(0,0)="O",
	"O"+<0cm,1.2cm>="11"*{{\bf 2}},
	"11"+<-0.3cm,-0.5cm>="21"*{},
	"21"+/r0.6cm/="22"*{1},
	
	\ar@{.}"11"+/dr/;"22"+/u/<-2pt>
	\end{xy},
	\hspace{15mm}
	\begin{xy} 
	(-13,7)*{Z=},
	(0,-0)="O",
	"O"+<0cm,1.2cm>="11"*{{\bf 1}},
	"11"+<-0.3cm,-0.5cm>="21"*{{\bf 2}},
	"21"+/r0.6cm/="22"*{},

	\ar@{-}"11"+/dl/;"21"+/u/<2pt>
	\end{xy}.
\]
One can check that $\stabHom_{A}(Y,X[-1])_0=\stabHom_{A}(Y,\Omega_A^1(X))_0\neq0$.
Therefore, $V$ is not a pretilting object. 

Note that $\stabgrCM A$ has a tilting object.
We will show the fact in Example \ref{example-truncated-pp-1} by using graded Morita equivalences induced from sink reflections of quivers.
\end{enumerate}
\end{example}

\subsection{Sufficient conditions for $A$ to satisfy $g(A)\leq0$}
In this subsection, we give a handy criterion for a graded $1$-IG-algebra $A$ to satisfy $g(A)\leq0$.

\begin{theorem}\label{main theorem 2}
Let $A = \bigoplus_{i= 0}^{\ell} A_{i}$ be a graded $1$-IG-algebra. 
Then, the following assertions hold.
\begin{enumerate}[\rm (1)]
\item 
If $\Hom_{A_{0}^{\op}}(A_{> 0}, A_{0} ) = 0$ holds, then we have $g(A)\leq 0$.
\item
Assume that $A_0$ has finite global dimension. 
If $\Hom_{A_{0}}(A_{> 0}, A_{0} ) = 0$ holds, then we have $g(A)\leq 0$.
\end{enumerate}
\end{theorem}

%
%

To prove the above, we give a necessary condition for asid bimodules that its asid numbers are equal to $2$. 

\begin{lemma}\label{necessary condition for alpha = 2}
Let $\Lambda$ be an algebra, $C$ an asid bimodule with $\alpha_r(C)=2$, and $A=\Lambda\oplus C$ the trivial extension algebra.
Assume that $A$ is a graded $1$-IG-algebra. 
Then, there is an idempotent $f$ of $\Lambda$ 
such that $Cf= 0$ and   the projective module $\Lambda f$ is a direct summand of $C$ as left $\Lambda$-modules.  
\end{lemma}

\begin{proof}
Let $ 0 \to A \to I^{0} \to I^{1} \to 0$ be a minimal graded injective resolution of $A_A$. 
The condition $\alpha_{r} (C)= 2$ implies that $(I^{1})_{-1}$ has a direct summand of the form  $f \tuD(\Lambda)$ for some idempotent $f \in \Lambda$. 
Observe that it follows from $I^{1}_{-2} = 0$ that $f\tuD(C) = 0$. 
On the other hand, it follows from $\soc (I^{0} )_{i} = 0$ for $i \neq 0,1$ that $I^{0}_{-1}$ belongs to $\add \tuD(C)$. 
 
Since $A_{-1} = 0$ the differential $I^{0} \to I^{1}$ gives an  isomorphism $(I^{0})_{-1} \xrightarrow{\cong} (I^{1})_{-1}$ of $\Lambda$-modules. 
From the above observations, we deduce that  $f\tuD(\Lambda)$ is a direct summand of $\tuD(C)$ as  right $\Lambda$-modules.   
\end{proof}

%

\begin{proof}[Proof of Theorem \ref{main theorem 2}]
(1)
We take the $\ell$-th quasi-Veronese algebra $A^{[\ell]}=\nabla A\oplus \Delta A$.
We denote by $\alpha_{r}$ the right asid number of the asid bimodule $\Delta A$ over $\nabla A$.
Note that $0\leq \alpha_r\leq 2$ by Corollary \ref{bound of asid number}.

Assume that $\Hom_{A_{0}^{\op}}(A_{> 0}, A_{0} ) = 0$.
In the following, we show that the condition $\alpha_{r} = 2$ leads to a contradiction. 
If $\alpha_{r}=2$, then the left $\nabla A$-modules $\nabla A$ and $\Delta A$ have a common indecomposable direct summand $P$ by Lemma \ref{necessary condition for alpha = 2}. 
Since $P$ is an indecomposable direct summand of the left $\nabla A$-module $\nabla A$, it is of the form  
\[
P = \begin{pmatrix} A_{i}f  \\ \vdots \\ A_{0}f \\0 \\ \vdots \\ 0 \end{pmatrix}
\]
where  $0\leq i \leq \ell - 1$ and  $f$ is a primitive idempotent of $A_{0}$.
Since $P$ is also a direct summand of the left $\nabla A$-module $\Delta A$, $A_{0}f$ is a direct summand of $A_{j}$ for some $j = 1, \cdots, \ell$  as left $A_{0}$-modules. This contradicts to the assumption $\Hom_{A_{0}^{\op}}(A_{> 0}, A_{0} ) = 0$. 
Thus we have $\alpha_r\leq1$, and so $g(A)\leq0$ by Lemma \ref{qv-g-asid}.

(2)
Assume that $\gldim A_0<\infty$ and $\Hom_{A_{0}}(A_{> 0}, A_{0} ) = 0$.
By the opposite version of (1), we have $g(A^{\op})\leq 0$. 
Then, $g(A)\leq 0$ also holds  by Corollary \ref{symmetry g(A)<=0}.
\end{proof}

We have the following result as a corollary of Theorem \ref{main theorem 2}.
It will be used later. 

\begin{corollary}\label{cor of main theorem 2}
Let $A = \bigoplus_{i= 0}^{\ell} A_{i}$ be a graded $1$-IG-algebra. 
Assume that $A$ satisfies the following conditions.
\begin{enumerate}[\em (i)]
\item 
$\gldim A_{0} = 1$. 
\item 
An $A_0$-module $A_{>0}$ does not contain projective $A_0$-modules as direct summands.
\end{enumerate}
Then, we have $g(A)\leq 0$.
\end{corollary}

\begin{proof}
The conditions (i) and (ii) imply that $\Hom_{A_{0}}(A_{> 0}, A_{0} ) = 0$.
Thus the assertion follows from Theorem \ref{main theorem 2}.
\end{proof}

\section{Application to the truncated preprojective algebras}\label{section-trucated-pp-alg}

In this section,  we study tilting theory for the truncated preprojective algebras introduced by \cite{BIRSc}.
We fix a finite acyclic connected quiver $Q=(Q_0,Q_1)$ with $Q_0=\{1,2,\cdots, n\}$.
For an arrow $\alpha$ going from $i$ to $j$, we write $s(\alpha)=i$ and $t(\alpha)=j$.
\[
\alpha : s(\alpha) \to t(\alpha)
\]
We denote by $|Q|$ the underlying graph of $Q$.

The double quiver $\overline{Q}$ of $Q$ is defined by $\overline{Q}_0:=Q_0$ and
\[
\overline{Q}_1:=Q_1 \sqcup \{ \alpha^*:t(\alpha)\to s(\alpha) \ | \ \alpha\in Q_1 \}.
\]
We call the $K$-algebra 
\begin{center}
$\Pi(Q):=K\overline{Q}/\langle \sum_{\alpha\in Q_1}(\alpha\alpha^*-\alpha^*\alpha) \rangle$
\end{center}
the \emph{preprojective algebra} of $Q$.
It is known that $\Pi(Q)$ is finite dimensional if and only if $|Q|$ is Dynkin \cite{Ringel}. 

We consider the grading on $\overline{Q}$ defined by
\[
\deg \alpha =\begin{cases}
0 & (\alpha \in Q_1) \\
1 & (\alpha \in\overline{Q}_1\backslash Q_1).
\end{cases}
\]
Then, $K\overline{Q}$ is a graded algebra with $(K\overline{Q})_0=KQ$.
Since an element $\sum_{\alpha\in Q_1}(\alpha\alpha^*-\alpha^*\alpha)$ of $K\overline{Q}$ is homogeneous of degree $1$, 
$\Pi(Q)$ is also a graded algebra with $\Pi(Q)_0=KQ$. 
We remark that if $Q'$ is an acyclic quiver with $|Q'|=|Q|$, then $\Pi(Q)$ and $\Pi(Q')$ are isomorphic to each other as algebras, but not isomorphic as graded algebras.

The \emph{Coxeter group} $W_Q$ of $Q$ is the group generated by
$\{s_i \ | \ i\in Q_0\}$ with relations $s_i^2=1$, $s_is_j=s_js_i$ if there exist no arrows between $i$ and $j$, and $s_is_js_i=s_js_is_j$ if there exists exactly one arrow between $i$ and $j$.
Let $w\in W_Q$.
An expression $w=s_{i_1}s_{i_2}\cdots s_{i_{r}}$ of $w$ is called \emph{reduced} if $r\leq k$ holds for any expression $w=s_{j_1}s_{j_2}\cdots s_{j_{k}}$ of $w$.
If a reduced expression of $w$ is given by $w=s_{i_1}s_{i_2}\cdots s_{i_{r}}$, then we define
\[
\supp(w):=\{i_1,i_2,\cdots,i_{r}\} \subset Q_0.
\]
This set $\supp(w)$ does not depend on the choice of a reduced expression of $w$ (see \cite[Corollary 1.4.8 (ii)]{BF}).

We define a graded two-sided ideal $I_w$ of $\Pi(Q)$ associated to $w\in W_Q$.
We denote by $e_i$ the idempotent of $\Pi(Q)$ corresponding to the vertex $i\in Q_0$.
For $i\in Q_0$, we define a graded two-sided ideal $I_i$ of $\Pi(Q)$ by
\[
I_i:=\Pi(Q)(1-e_i)\Pi(Q).
\]
For  $w\in W_Q$ with a reduced expression $w=s_{i_1}s_{i_2}\cdots s_{i_{r}}$, 
we define a graded two-sided ideal $I_w$ of $\Pi(Q)$ by
\[
I_w:=I_{i_1}I_{i_2}\cdots I_{i_{r}}.
\]
Note that $I_w$ does not depend on the choice of a reduced expression of $w$ (see \cite[Theorem II.1.9]{BIRSc}).

Now we define a graded algebra $\Pi(Q)_w$ by
\[
\Pi(Q)_w:=\Pi(Q)/I_w.
\]
This is a finite dimensional graded $1$-IG-algebra \cite[Proposition II.2.2]{BIRSc}. 
We call $\Pi(Q)_w$ \emph{the truncated preprojective algebra of $Q$ associated to $w$}.
The aim of this section is to prove the following result. 

\begin{theorem}\label{tilting theorem}
Let $Q$ be a finite acyclic connected  quiver, and $w\in W_Q$.
If $|Q|$ is tree, 
then $\stabgrCM \Pi(Q)_{w}$ has a tilting object. 
\end{theorem}

\subsection{Sink reflections and graded Morita equivalences}
In the proof of Theorem \ref{tilting theorem}, sink reflections of quivers play an essential role. 
We recall the definition of sink reflections of quivers.

\begin{definition}
Let $i\in Q_0$ be a sink vertex, i.e. there are no arrows going from the vertex $i$.
The \emph{sink reflection of $Q$ at the vertex $i$} is the quiver $\sigma_i(Q)$ defined by $(\sigma_i(Q))_0:=Q_0$ and 
\[
\sigma_i(Q)_1:= \{\alpha \in Q_1 \ | \ t(\alpha)\neq i \} \, \sqcup\, \{ \alpha' :i \to s(\alpha) \ | \  \mbox{$\alpha\in Q_1$ s.t. $t(\alpha)=i$}\}.
\]
In brief, $\sigma_i(Q)$ is the quiver obtained by reversing all arrows of $Q$ going to the vertex $i$.
\end{definition}

\begin{remark}\label{reflection-Coxeter grp}
Let $i\in Q_0$ be a sink vertex.
Then we have $\Pi(Q)\cong\Pi(\sigma_i(Q))$ as ungraded algebras since $|Q|=|\sigma_i(Q)|$.
Moreover, we can identify $W_Q$ and $W_{\sigma_i(Q)}$ by the following correspondence:
\[
W_Q \ni s_{i_1} s_{i_2} \cdots s_{i_{r}} \leftrightarrow s_{i_1} s_{i_2} \cdots s_{i_{r}} \in W_{\sigma_i(Q)}.
\]
\end{remark}

Now, we prove that the category of all graded $\Pi(Q)_w$-modules is invariant under sink reflections.

\begin{proposition}\label{reflection proposition}
Let $i\in Q_0$ be a sink vertex, and $Q':=\sigma_i(Q)$.
We take $w \in W_{Q}=W_{Q'}$,
and set $A=\Pi(Q)_w$ and $B=\Pi(Q')_{w}$.
Then, the following assertions hold.
\begin{enumerate}[\rm (1)]
\item 
Let $P:=e_{i} A(1) \oplus (A/e_iA)\in\grmod A$. 
Then, we have $\End_{A}(P)\cong B$ as graded algebras.
\item
The graded algebras $A$ and $B$ are graded Morita equivalent.
Therefore, we have an equivalence 
\[
\stabgrCM \Pi(Q)_{w} \cong \stabgrCM \Pi(Q')_{w}
\]
of triangulated categories.
\end{enumerate}
\end{proposition}

\begin{proof}
(1)
We prove that there is an isomorphism $B\to \End_{A}(P)$ of graded algebras.
First $A$ and $B$ are isomorphic to each other as ungraded algebras. 
In fact, a pair of maps $\phi_0:Q'_0\to \Pi(Q)$ and  $\phi_1:Q'_1\to \Pi(Q)$ defined by 
\[
\phi_0(i):=e_i, \hspace{5mm}
\phi_1(\beta)=
\begin{cases}
-\alpha^* &  \mbox{ if }  \beta=\alpha' \mbox{ where } \alpha\in Q_1 \mbox{ with } t(\alpha)= i, \\
\alpha  &  \mbox{ if }  \beta=(\alpha')^* \mbox{ where } \alpha\in Q_1 \mbox{ with } t(\alpha)= i,\\
\beta  & \mbox{ if } s(\beta)\neq i \mbox{ and } t(\beta)\neq i. \\
\end{cases}
\]
is extended to an isomorphism $\phi: \Pi(Q')\to \Pi(Q)$ of ungraded algebras \cite[Chapter II. Theorem 1.8]{Assem-Simson-Skowronski}.
It is easy to see that the map $\phi$ sends $I'_{w}$ to $I_{w}$ where $I'_w$ is the ideal of $\Pi(Q')$ associated to $w$. 
Hence,  $\phi$ induces  an isomorphism $\widetilde{\phi}: B\to A$ of ungraded algebras. 

Next  we consider the following composite of homomorphisms of ungraded algebras:
\[
B \xrightarrow{\widetilde{\phi}} A \xrightarrow{\mathsf{lm}} \End_{A}(P)  
\] 
where $\mathsf{lm}$ is the map which sends a homogeneous element  $a \in A$ to the left multiplication map $a\, \cdot  : P \to P$. 
It is easy to check that this composite is an isomorphism of graded algebras.  

(2)
If we regard $P$ as an ungraded $A$-module, then it is a projective generator of $\mod A$. 
Thus, $A$ and $B$ are graded Morita equivalent by (1) and \cite[Theorem 5.4]{Green-Gordon1}.
\end{proof}


%

\subsection{Admissible sequences of quivers and the existence of tilting objects}

We show that $\stabgrCM \Pi(Q)_w$ has a tilting object if an admissible sequence of $Q$ appears in the sequence of indexes of a reduced expression of $w$.
We recall the definition of admissible sequences in $Q$.

\begin{definition}
%
A sequence $(i_1,i_2,\cdots,i_n)$ of all vertices $1,2,\cdots,n$ in $Q$ is called an \emph{admissible sequence in $Q$}
if $j<k$ whenever there is an arrow going from $i_j$ to $i_k$.
\end{definition}

Since $Q$ is acyclic, an admissible sequences in $Q$ exists.
Note that a sequence $(i_1,i_2,\cdots,i_n)$ of all vertices in $Q$ is admissible in the sense of the above definition if and only if it is an admissible sequence of sources in the sense of \cite[Chapter VII.5.]{Assem-Simson-Skowronski}.



\begin{theorem}\label{special-case}
Let $w\in W_Q$ with a reduced expression $w=s_{i_1} s_{i_2} \cdots s_{i_{r}}$.
Assume that the sequence $(i_1,i_2,\cdots,i_{r})$ of $Q_0$ contains an admissible sequence in $Q$ as a subsequece.
Then, we have $g(\Pi(Q)_w)\leq 0$, and so $V$ is a tilting object in  $\stabCM^{\ZZ}\Pi(Q)_w$.
\end{theorem}

\begin{proof}
We set $A=\Pi(Q)_w$. 
By the assumption and \cite[Lemma 3.2]{Kimura 1}, an inclusion relation $I_w \subset \Pi(Q)_{>0}$ holds. 
Therefore, we have $A_0=\Pi(Q)_0=KQ$, and so $\gldim A_0= 1$.
Since $A_i$ is a factor $KQ$-module of $\Pi(Q)_i\cong \tau^{-i}(KQ)$ where $\tau$ is the Auslander-Reiten translation of $\mod KQ$,  $A_i$ does not contain projective $KQ$-modules as direct summands. 
Then, we have $g(A)\leq0$ by Corollary \ref{cor of main theorem 2}.
Thus, $\stabgrCM A$ has a tilting object $V$ by Theorem \ref{tilting theorem 1}.
\end{proof}

\begin{remark}
The first author proved that $\stabgrCM \Pi(Q)_w$ has a tilting object if $w$ satisfies some conditions which is stronger than the assumption of Theorem \ref{special-case} (see \cite[Theorem 1.2 and Lemma 4.4]{Kimura 2}).
His tilting object is different from our tilting object $V$.
\end{remark}

We illustrate how to apply Proposition \ref{reflection proposition} and Theorem \ref{special-case}.

\begin{example}\label{example-truncated-pp-1}
Let $Q$ be a quiver
\[
\begin{xy}
	(0,10)="O",
	"O"+<0cm,0cm>="1"*{1},
	"1"+<-0.8cm,-1.2cm>="2"*{2},
	"1"+<0.8cm,-1.2cm>="3"*{ \ 3,},
	\ar "1"+<-0.1cm,-0.2cm>;"2"+<0.2cm,0.2cm>_{\alpha}
	\ar "1"+<0.1cm,-0.2cm>;"3"+<-0.2cm,0.2cm>^{\gamma}
	\ar "2"+<0.2cm,-0cm>;"3"+<-0.2cm,0cm>_{\beta}
\end{xy}
\]
and $w=s_2s_3s_1s_3s_2s_1\in W_Q$.
We set $A=\Pi(Q)_w$.
Note that this algebra was considered in Example \ref{example-g(A)=1} (2).
It was shown that $g(A)=1$ holds and $V$ is not a tilting object in $\stabgrCM A$. 
In the following, we construct a tilting object in $\stabgrCM A$ by applying Proposition \ref{reflection proposition} and Theorem \ref{special-case}.

We mention that the sequence $(2,3,1,3,2,1)$ of indexes of the reduced expression $w=s_2s_3s_1s_3s_2s_1$ does not contain a unique admissible sequence $(1,2,3)$ in $Q$.
Therefore, we can not apply Theorem \ref{special-case} to $A$.
Now we consider the sink reflection $Q':=\sigma_3(Q)$ of $Q$, that is, 
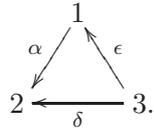
\[
\begin{xy}
	(0,10)="O",
	"O"+<0cm,0cm>="1"*{1},
	"1"+<-0.8cm,-1.2cm>="2"*{2},
	"1"+<0.8cm,-1.2cm>="3"*{ \ 3.},
	\ar "1"+<-0.1cm,-0.2cm>;"2"+<0.2cm,0.2cm>_{\alpha}
	\ar@{<-} "1"+<0.1cm,-0.2cm>;"3"+<-0.2cm,0.2cm>^{\epsilon}
	\ar@{<-} "2"+<0.2cm,-0cm>;"3"+<-0.2cm,0cm>_{\delta}
\end{xy}
\]
Let $A'=\Pi(Q')_w$.
Then, the sequence $(2,3,1,3,2,1)$ contains a unique admissible sequence $(3,1,2)$ in $Q'$ as a subsequence. 
By Theorem \ref{special-case}, $A'$ satisfies $g(A')\leq 0$, and so $V'=\bigoplus_{i>0}A'(i)_{\geq0}$ is a tilting object in $\stabgrCM A'$.
Since there is an equivalence $\stabgrCM A \cong \stabgrCM A'$ of triangulated categories by Theorem \ref{reflection proposition}, $\stabgrCM A$ has a tilting object.

We calculate $V'$ and its endomorphism algebra. 
The algebras $A$ and $A'$ are isomorphic to each other as ungraded algebras, but those gradings are different. 
The graded quiver with relations of $A'$ is obtained by that of $A$ and the proof of Theorem \ref{reflection proposition} (compare with Example \ref{example-g(A)=1} (2)).
\[
\begin{xy}
	(0,10)="O",
	"O"+<0cm,0cm>="1"*{1},
	"1"+<-1cm,-1.5cm>="2"*{2},
	"1"+<1cm,-1.5cm>="3"*{3},
	\ar "1"+<-0.1cm,-0.2cm>;"2"+<0.2cm,0.2cm>^{\alpha}
	\ar@/^/ "2"+<0cm,0.2cm>;"1"+<-0.3cm,-0.1cm>^{\alpha^*}
	\ar "1"+<0.1cm,-0.2cm>;"3"+<-0.2cm,0.2cm>_{\epsilon^*}
	\ar@/_/ "3"+<0cm,0.2cm>;"1"+<0.3cm,-0.1cm>_{\epsilon}
	\ar "2"+<0.2cm,-0cm>;"3"+<-0.2cm,0cm>^{\delta^*}
	\ar@/^/ "3"+<-0.2cm,-0.2cm>;"2"+<0.2cm,-0.2cm>^{\delta}
\end{xy} 
\hspace{5mm}
\begin{cases}
\deg \alpha=\deg\delta=\deg\epsilon=0 \\
\deg \alpha^*=\deg \delta^*=\deg\epsilon^*=1,
\end{cases}
\mbox{ with } 
\hspace{5mm}
\begin{cases}
\alpha\alpha^*-\epsilon^*\epsilon=0 \\
\alpha^*\alpha+\delta^*\delta=0 \\
\epsilon\epsilon^*=\delta\delta^*=0 \\
\epsilon\alpha\delta^*=0\\
\delta\alpha^*\epsilon^*=0 \\
\alpha\delta^*=0
\end{cases}
\]
The Loewy series of indecomposable graded projective $A'$-modules $e_1A',e_2A',e_3A'$ are as follows.
\[
	\begin{xy} 
	(-15,2)*{e_1A'=},
	(0,0)="O",
	"O"+<0cm,1.2cm>="11"*{{\bf 1}},
	"11"+<-0.3cm,-0.5cm>="21"*{{\bf 2}},
	"21"+<-0.3cm,-0.5cm>="31"*{},
	"31"+<-0.3cm,-0.5cm>="41"*{},
	"41"+<-0.3cm,-0.5cm>="51"*{},
	"21"+/r0.6cm/="22"*{3},
	"31"+/r0.6cm/="32"*{1},
	"32"+/r0.6cm/="33"*{2},
	"41"+/r0.6cm/="42"*{},
	"42"+/r0.6cm/="43"*{},
	"43"+/r0.6cm/="44"*{1},
	"51"+/r0.6cm/="52"*{},
	"52"+/r0.6cm/="53"*{},
	"53"+/r0.6cm/="54"*{},
	"54"+/r0.6cm/="55"*{},
	
	\ar@{-}"11"+/dl/;"21"+/u/<2pt>
	\ar@{.}"11"+/dr/;"22"+/u/<-2pt>
	\ar@{.}"21"+/dr/;"32"+/u/<-2pt>
	\ar@{-}"22"+/dl/;"32"+/u/<2pt>
	\ar@{-}"22"+/dr/;"33"+/u/<-2pt>
	\ar@{.}"33"+/dr/;"44"+/u/<-2pt>
	\end{xy},
\hspace{10mm}
	\begin{xy} 
	(-17,2)*{e_2A'=},
	(0,0)="O",
	"O"+<0cm,1.2cm>="11"*{{\bf 2}},
	"11"+<-0.3cm,-0.5cm>="21"*{3},
	"21"+<-0.3cm,-0.5cm>="31"*{1},
	"31"+<-0.3cm,-0.5cm>="41"*{2},
	"41"+<-0.3cm,-0.5cm>="51"*{},
	"21"+/r0.6cm/="22"*{1},
	"31"+/r0.6cm/="32"*{2},
	"32"+/r0.6cm/="33"*{3},
	"41"+/r0.6cm/="42"*{},
	"42"+/r0.6cm/="43"*{1},
	"43"+/r0.6cm/="44"*{2},
	"51"+/r0.6cm/="52"*{},
	"52"+/r0.6cm/="53"*{},
	"53"+/r0.6cm/="54"*{},
	"54"+/r0.6cm/="55"*{1},

	\ar@{.}"11"+/dl/;"21"+/u/<2pt>
	\ar@{.}"11"+/dr/;"22"+/u/<-2pt>
	\ar@{-}"21"+/dl/;"31"+/u/<2pt>
	\ar@{-}"21"+/dr/;"32"+/u/<-2pt>
	\ar@{-}"22"+/dl/;"32"+/u/<2pt>
	\ar@{.}"22"+/dr/;"33"+/u/<-2pt>
	\ar@{-}"31"+/dl/;"41"+/u/<2pt>
	\ar@{.}"32"+/dr/;"43"+/u/<-2pt>
	\ar@{-}"33"+/dl/;"43"+/u/<2pt>
	\ar@{-}"33"+/dr/;"44"+/u/<-2pt>
	\ar@{.}"44"+/dr/;"55"+/u/<-2pt>
	\end{xy},
\hspace{10mm}
	\begin{xy} 
	(-17,2)*{e_3A'=},
	(0,0)="O",
	"O"+<0cm,1.2cm>="11"*{{\bf 3}},
	"11"+<-0.3cm,-0.5cm>="21"*{1},
	"21"+<-0.3cm,-0.5cm>="31"*{2},
	"31"+<-0.3cm,-0.5cm>="41"*{},
	"41"+<-0.3cm,-0.5cm>="51"*{},
	"21"+/r0.6cm/="22"*{2},
	"31"+/r0.6cm/="32"*{},
	"32"+/r0.6cm/="33"*{1},
	"41"+/r0.6cm/="42"*{},
	"42"+/r0.6cm/="43"*{},
	"43"+/r0.6cm/="44"*{},
	"51"+/r0.6cm/="52"*{},
	"52"+/r0.6cm/="53"*{},
	"53"+/r0.6cm/="54"*{},
	"54"+/r0.6cm/="55"*{},
	
	\ar@{-}"11"+/dl/;"21"+/u/<2pt>
	\ar@{-}"11"+/dr/;"22"+/u/<-2pt>
	\ar@{-}"21"+/dl/;"31"+/u/<2pt>
	\ar@{.}"22"+/dr/;"33"+/u/<-2pt>
	\end{xy}.
\]
Our $V'=\Omega_A^1(U')$ is given by $V'=X^{\oplus2}\oplus Y \oplus Z^{\oplus3}$ where $X,Y,Z$ are as follows.
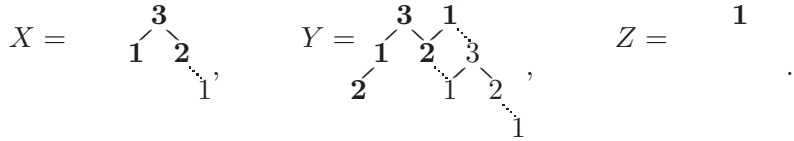
\[
	\begin{xy} 
	(-13,4)*{X=},
	(0,0)="O",
	"O"+<0cm,1.2cm>="11"*{},
	"11"+<-0.3cm,-0.5cm>="21"*{},
	"21"+<-0.3cm,-0.5cm>="31"*{},
	"31"+<-0.3cm,-0.5cm>="41"*{},
	"41"+<-0.3cm,-0.5cm>="51"*{},
	"21"+/r0.6cm/="22"*{{\bf 3}},
	"31"+/r0.6cm/="32"*{{\bf 1}},
	"32"+/r0.6cm/="33"*{{\bf 2}},
	"41"+/r0.6cm/="42"*{},
	"42"+/r0.6cm/="43"*{},
	"43"+/r0.6cm/="44"*{1},

	\ar@{-}"22"+/dl/;"32"+/u/<2pt>
	\ar@{-}"22"+/dr/;"33"+/u/<-2pt>
	\ar@{.}"33"+/dr/;"44"+/u/<-2pt>
	\end{xy},
	\hspace{10mm}
	\begin{xy} 
	(-13,4)*{Y=},
	(0,0)="O",
	"O"+<0cm,1.2cm>="11"*{},
	"11"+<-0.3cm,-0.5cm>="21"*{{\bf 3}},
	"21"+<-0.3cm,-0.5cm>="31"*{{\bf 1}},
	"31"+<-0.3cm,-0.5cm>="41"*{{\bf 2}},
	"41"+<-0.3cm,-0.5cm>="51"*{},
	"21"+/r0.6cm/="22"*{{\bf 1}},
	"31"+/r0.6cm/="32"*{{\bf 2}},
	"32"+/r0.6cm/="33"*{3},
	"41"+/r0.6cm/="42"*{},
	"42"+/r0.6cm/="43"*{1},
	"43"+/r0.6cm/="44"*{2},
	"51"+/r0.6cm/="52"*{},
	"52"+/r0.6cm/="53"*{},
	"53"+/r0.6cm/="54"*{},
	"54"+/r0.6cm/="55"*{1},

	\ar@{-}"21"+/dl/;"31"+/u/<2pt>
	\ar@{-}"21"+/dr/;"32"+/u/<-2pt>
	\ar@{-}"22"+/dl/;"32"+/u/<2pt>
	\ar@{.}"22"+/dr/;"33"+/u/<-2pt>
	\ar@{-}"31"+/dl/;"41"+/u/<2pt>
	\ar@{.}"32"+/dr/;"43"+/u/<-2pt>
	\ar@{-}"33"+/dl/;"43"+/u/<2pt>
	\ar@{-}"33"+/dr/;"44"+/u/<-2pt>
	\ar@{.}"44"+/dr/;"55"+/u/<-2pt>
	\end{xy},
	\hspace{10mm}
	\begin{xy} 
	(-13,4)*{Z=},
	(0,0)="O",
	"O"+<0cm,0.7cm>="11"*{{\bf 1}},
	"11"+<-0.3cm,-0.5cm>="21"*{},
	"21"+<-0.3cm,-0.5cm>="31"*{},
	"31"+<-0.3cm,-0.5cm>="41"*{},
	"41"+<-0.3cm,-0.5cm>="51"*{},
	"21"+/r0.6cm/="22"*{},
	"31"+/r0.6cm/="32"*{},
	"32"+/r0.6cm/="33"*{},
	\end{xy}.
\]
The endomorphism algebra $\underline{\End}_{A'}(V')_0$ is Morita equivalent to the algebra $\Gamma$ defined by the following quiver with the relation:
\[
\xymatrix{
\bullet \ar[r]^a &\bullet \ar[r]^b & \bullet
}
\hspace{10mm}
ab=0.
\] 
Therefore, there is an equivalences $\stabgrCM A \cong \stabgrCM A' \cong\sfD^{\mrb}(\mod\Gamma)$ of triangulated categories.
\end{example}

\subsection{Proof of Theorem \ref{tilting theorem}}

Now, we give a proof of Theorem \ref{tilting theorem}.
We need one more lemma. 

\begin{lemma}[{\cite[Lemma 2.1]{AIRT}}]\label{rejection lemma}
Let $w\in W_Q$, and $\widetilde{Q}$ the full subquiver of $Q$ such that $\widetilde{Q}_0=\supp(w)$. 
We regard $w$ is an element in $W_{\widetilde{Q}}\subset W_Q$.
Then, we have $\Pi(Q)_{w} = \Pi(\widetilde{Q})_{w}$ as graded algebras.
\end{lemma}

\begin{proof}[Proof of Theorem \ref{tilting theorem}]
The proof consists of three steps. 
\begin{enumerate}[(Step 1)]
\item
Thanks to Lemma \ref{rejection lemma}, we may assume $Q_0=\supp (w)$.
Now, we fix a reduced expression $w=s_{i_1}s_{i_2}\cdots s_{i_{r}}$ of $w$.
The sequence $(i_1,i_2,\cdots,i_{r})$ of $Q_0$ contains a sequence $(i_{j_1},i_{j_2},\cdots,i_{j_n})$ of all vertices in $Q$.
\item
Since $|Q|$ is tree, it is easy to see that there is an acyclic quiver $Q'$ such that $|Q|=|Q'|$ and $(i_{j_1},i_{j_2},\cdots,i_{j_n})$ is an admissible sequence in $Q'$.
By \cite[Chapter VII. Lemma 5.2]{Assem-Simson-Skowronski}, 
we can obtain $Q'$ as an iterated sink reflections of $Q$.
Since $W_{Q}=W_{Q'}$ as in Remark \ref{reflection-Coxeter grp}, we regard $w$ as an element of $W_{Q'}$. 
Then, by Proposition \ref{reflection proposition}, there is an equivalence
\[
\stabgrCM \Pi(Q)_{w} \cong \stabgrCM \Pi(Q')_{w}
\]
of triangulated categories.
\item
Since $\stabgrCM \Pi(Q')_{w}$ has a tilting object by Theorem \ref{special-case},
so does $\stabgrCM \Pi(Q)_{w}$.
\qedhere
\end{enumerate}
\end{proof}

\end{document}